\documentclass[stsy,non-blindrev]{informs3}

\OneAndAHalfSpacedXI

\usepackage{natbib}
\bibpunct[, ]{(}{)}{,}{a}{}{,}%
\def\bibfont{\small}%
\TheoremsNumberedBySection  
\ECRepeatTheorems

\EquationsNumberedBySection 

\MANUSCRIPTNO{MS-0001-1922.65}

\usepackage{amstext,amsfonts,bm,amssymb}
\usepackage{mathrsfs}
\usepackage{amsfonts}
\usepackage{amsmath,amssymb}
\usepackage[shortlabels]{enumitem}
\usepackage{tabularx}
\usepackage{float, graphicx, subfigure}
\usepackage{hyperref}
\usepackage{multirow}
\usepackage{rotating}
\usepackage{verbatim}
\usepackage{color, soul}
\usepackage{blkarray, bigstrut}
\usepackage{mathtools}
\usepackage[normalem]{ulem}
\usepackage[makeroom]{cancel}

\allowdisplaybreaks

\def\ohad#1{\textcolor{red}{#1}}

\newcommand{\RR}{{\mathbb R}}
\newcommand{\ZZ}{{\mathbb Z}}
\newcommand{\NN}{{\mathbb N}}

\newcommand{\lm}{\lambda}
\newcommand{\ep}{\epsilon}

\newcommand{\ra}{\rightarrow}
\newcommand{\deq}{\stackrel{\rm d}{=}}

\newcommand{\qandq}{\quad\mbox{and}\quad}

\newcommand{\qforq}{\quad\mbox{for }}

\newcommand{\qforallq}{\quad\mbox{for all }}

\def\tinf{\rightarrow\infty}

\def\Ra{\Rightarrow}

\newcommand{\bes}{\begin{equation*}}
\newcommand{\ees}{\end{equation*}}
\newcommand{\bequ}{\begin{equation}}
\newcommand{\eeq}{\end{equation}}
\newcommand{\bsplit}{\begin{split}}
\newcommand{\esplit}{\end{split}}
\newcommand{\bea}{\begin{eqnarray}}
\newcommand{\eea}{\end{eqnarray}}
\newcommand{\beas}{\begin{eqnarray*}}
\newcommand{\eeas}{\end{eqnarray*}}
\newcommand{\benum}{\begin{enumerate}}
\newcommand{\eenum}{\end{enumerate}}

\newcommand{\RNum}[1]{\uppercase\expandafter{\romannumeral #1\relax}}

\newcommand{\E}{\mathbb{E}}
\newcommand{\R}{\mathbb{R}}

\newcommand{\of}[1]{\left(#1\right)}
\newcommand{\off}[1]{\left[#1\right]}
\newcommand{\offf}[1]{\left\{#1\right\}}
\newcommand{\arr}{\rightarrow}
\newcommand{\Arr}{\Rightarrow}

\RUNTITLE{Moments for Polling Systems}

\TITLE{\Large Existence and Approximations of Moments for Polling Systems under the Binomial-Exhaustive Policy}

\ARTICLEAUTHORS{%
	\AUTHOR{Yue Hu}
	\AFF{Decision, Risk, and Operations, Columbia Business School;  \EMAIL{yhu22@gsb.columbia.edu}}
	\AUTHOR{Jing Dong}
	\AFF{Decision, Risk, and Operations, Columbia Business School; \EMAIL{jing.dong@gsb.columbia.edu}}
	\AUTHOR{Ohad Perry}
	\AFF{Department of Industrial Engineering and Management Sciences, Northwestern University; \EMAIL{ohad.perry@northwestern.edu}}
} 

\ABSTRACT{We establish sufficient conditions for the existence of moments of the steady-state queue in polling systems operating under the binomial-exhaustive policy (BEP). We assume that the server switches between the different buffers according to a pre-specified table, and that switchover times are incurred whenever the server moves from one buffer to the next. We further assume that customers arrive according to independent Poisson processes, and that the service and switchover times are independent random variables with general distributions. We then propose a simple scheme to approximate the moments, which is shown to be asymptotically exact as the switchover times increase without bound, and whose computation complexity does not grow with the order of the moment. Finally, we demonstrate that the proposed asymptotic approximations for the moments is related to the fluid limit under a large-switchover-time scaling; thus, similar approximations can be easily derived for other server-switching policies, by simply identifying the fluid limits under those controls.
Numerical examples demonstrate the effectiveness of our approximations for the moments under BEP and other policies,
and their improving accuracy as the switchover times increase.}


\begin{document}
\maketitle



\section{Introduction}
Polling systems are queueing networks in which multiple queues are attended by
a single server that switches between the different buffers according to a given switching policy.
This class of systems is used to model a large variety of systems in practice, such as
communication, production, transportation, healthcare and computer systems; see, e.g., \cite{levy1990polling,takagi1991application,federgruen1996stochastic,takagi1997queueing,cicin2001application,boon2011applications,borst2018polling}.

In most of the literature, rigorous analysis of stochastic polling systems under a given switching policy is
carried out by deriving multi-dimensional transforms for the queue process at the polling epochs,
from which the distribution of the queue can, in principle, be computed by inverting the transform. However, as explained in \cite{choudhury1996computing},
those transforms are not available directly, and are instead expressed implicitly in a recursive form.
Even more fundamental is the fact that transforms can only be computed for a class of policies that satisfy a certain branching-type property \citep{resing1993polling}. Nevertheless, algorithms to invert the transforms, when they can be derived, are known, and can be used
to numerically compute distributions and moments.

Arguably, the most important moments are the first two, but both the question of whether higher moments exist, and the values of such moments
are also of interest for at least two reasons: First, higher moments can be used to evaluate tail probabilities of the waiting times
in the different queues \citep{choudhury1996computing}.
Second, the existence of the $p$th moment, say, can be used to prove uniform integrability of the $m$th power of the queue, $m < p$,
which in turn can be used to evaluate the cost of a policy when a holding cost (that grows at most at a polynomial rate) is incurred on the queues
via asymptotic approximations; see \cite{Hu20arxiv} for an example.
However, for most policies for which computational algorithms to evaluate the moments have been developed,
it has not been established that the moments, other than the mean, exist.

In this paper we provide sufficient conditions for the existence of all the moments, and necessary-and-sufficient conditions
for the existence of the first two moments of the queue process at polling epochs, when the system operates under the binomial-exhaustive policy (BEP) \citep{levy1988optimization,levy1990dominance}.
To this end, we employ the buffer occupancy approach, first developed in \cite{cooper1969queues}.
We then prove that the moments converge to the (unique) solution to the buffer occupancy equations under a large switchover times asymptotic, namely,
as the switchover times increase without bound. Thus, the solution to the buffer occupancy equations can be used to approximate any of the moments,
provided the switchover times are large enough.
Finally, we make the observation that the buffer occupancy equations are also the balance equations of the periodic steady state
of a fluid model for the system, implying that the $p$th moments of the queue process at polling epochs can be approximated by taking
the $p$th power of the fluid model at the corresponding time epochs. Since fluid models are easy to derive, this fluid-type approximation
can be used as a heuristic even when the asymptotic approximations for the moments cannot be proved rigorously.
Our numerical examples demonstrate the good performance of this heuristic.

\subsection{Background and Related Literature}
The buffer occupancy approach establishes a recursive equation for the multi-dimensional probability generating function (p.g.f.)
for the number in system at the polling epochs.
Differentiation of the p.g.f.\ yields systems of linear equations, referred to as the buffer occupancy equations,
satisfied by the moments of the number in system, whenever these moments exist.
This approach has been applied to analyzing a wide variety of polling systems,
including cyclic systems with zero switchover times \cite{cooper1969queues, cooper1970queues};
cyclic systems with nonzero switchover times \cite{konheim1974waiting, rubin1983message}; and non-cyclic polling systems \cite{kleinrock1988analysis}.
We refer to \cite{takagi1986analysis} for applications of the approach for systems with cyclic routing operating under the exhaustive and gated policies.

It is significant that solving the buffer occupancy equations can be numerically challenging.
In the simplest setting, in which $K$ queues are visited in a cyclic order and the exhaustive policy is employed,
the set of buffer occupancy equations for the $p$th moment consists of $K^{p+1}$ linear equations.
To make the calculation computationally feasible, several numerical methods have been proposed to efficiently solve the buffer occupancy equations.
In particular, \cite{levy1989delay} proposes a successive approximation approach to solve the second-order buffer occupancy equations.
For this solution method, the author shows that solving the $K^3$ equations requires $O(K^3 \log_\rho \ep)$ elementary operations,
where $\rho$ is the system utilization and $\ep$ is the specified relative accuracy.
Variants of the buffer occupancy approach include the station time approach \citep{aminetzah1975exact,humblet1978source,ferguson1986computation,baker1987polling},
Swartz's approach \citep{swartz1980polling}, Sarkar and Zangwill's approach \citep{sarkar1989expected},
the descendant set method \citep{konheim1994descendant}, and the numerical inverse transform algorithm \citep{choudhury1996computing}.
only few works establish that the buffer occupancy equations admit a unique solution.
For systems with cyclic routing under the exhaustive and gated policies, \cite{takagi1986analysis} establishes uniqueness of the solution to the first-order equations by directly solving those equations; \cite{levy1989delay} proves the existence of a unique fixed point of the second-order equations via a contraction mapping argument.

We emphasize three fundamental points:
First, characterizing the buffer occupancy equations of order higher than two is prohibitive.
Second, a numerical solution to these equations does not imply that a unique solution exists.
Finally, it is not immediately clear that the existence of a unique solution to the equations
guarantees that the moments themselves exist.
We remark that we address the final point by proving that the existence of a unique solution to the buffer occupancy equations
(under any control for which these equations are available) does imply that the moments exist; see Theorem \ref{lem:SecondMoment} and Remark \ref{rem:SecondMoment}.
This latter result can be used to compute and to guarantee the existence of the first few moments whenever the buffer occupancy equations are available.

\subsection{Summary of Our Contribution}
We consider a system with $K$ queues, that are fed by $K$ independent Poisson processes, and are attended by the server in accordance with a table consisting of $I$ stages,
$I \geq K$. The server switches from one stage to the next according to BEP \citep{levy1988optimization,levy1990dominance},
which belongs to the family of branching-type controls \citep{resing1993polling,winands2011branching}, and for which the exhaustive policy is a special case.
A random switchover time is incurred whenever the server switches.
In this setting we establish the following results.

\noindent{\bf Sufficient condition for the existence of all moments.~}
We apply the buffer occupancy approach to derive the p.g.f.\ of the steady-state queue length at the polling epochs. Based on that p.g.f.,
we prove that, if the service time distribution at each queue has a moment generating function (m.g.f.) that is finite in some neighborhood of zero,
and if the switchover time distribution at each stage has finite m.g.f.\ on the positive real line,
then the steady-state queue length at the polling epoch of each stage has finite moments of all orders.

\noindent{\bf Necessary and sufficient condition for the existence of the second moment.~}
Differentiating the p.g.f.\ of the steady-state queue length $p$ times yields the $p$th-order buffer occupancy equations---a system of $IK^p$ equations satisfied by the $p$th moment and cross moment of the steady-state queue length.
Using a contraction argument, we show that the first- and second-order buffer occupancy equations each admit a unique solution.
Moreover, we prove that a necessary-and-sufficient condition for the steady-state queue length distribution to have finite second moment
is for the variance of each service time and switchover time distribution to exist.

\noindent{\bf Asymptotic approximations for the moments.~} Under a large switchover-time scaling, namely, assuming that the switchover times
increase without bound, we prove that the $p$th moments of the ``fluid-scaled'' queues at polling epochs converge to the solution to the first-order buffer occupancy equations, raised to the $p$th power. This result
suggests a simple approximation for the moments using fluid models.
In particular, we make the observation that the first-order buffer occupancy equations are equivalent to the balance equations of a deterministic fluid model
for the system in steady state (more accurately, in a periodic steady state). Further, that fluid model arises as a bona-fide
functional weak law of large numbers (FWLLN) under the large switchover-times scaling.

\medskip

In ending we emphasize that, even though our focus is on BEP, the main arguments in the paper can be repeated for other policies for which the buffer occupancy approach applies.
Moreover, the fluid-model approach can be used to approximate the moments for policies in general, including those to which the buffer occupancy approach is not applicable.
Although we do not carry out the full rigorous analysis for other policies,
we demonstrate the effectiveness of our approximations for the moments under the binomial-gated policy (BGP), and base-stock policy (BSP)
via numerical examples; see Sections \ref{sec:OtherPolicy}. Note that BSP is not a branching-type policy,
and so the transforms for the queue process at polling epochs are not known.

\subsection{Notation}
All the random variables and processes are defined on a single probability space $(\Omega, \mathcal F, \mathbb{P})$; expectation with respect to $\mathbb P$ is denote by $\E$.
We let $\RR$, $\ZZ$ and $\NN$ denote the sets of real numbers, integers and strictly positive integers, respectively.
We also write $\ZZ_+ := \NN \cup \{0\}$ and $\RR_+ := [0,\infty)$. For $k \in \NN$, we let
$\RR^k$ denote the space of $k$-dimensional vectors with real components, and denote these vectors with bold letters and numbers; in particular,
we write $\boldsymbol 1 := (1,\dots, 1)$ for the vector of $1$'s.
We let $D^k$ denote the space of right-continuous $\RR^k$-valued functions (on arbitrary finite time intervals) with limits everywhere,
endowed with the usual Skorokhod $J_1$ topology; see Chapter 11 of \cite{whitt2002stochastic}.
We let $D := D^1$.
We use $C^k$ (and $C := C^1$) to denote the subspace of $D^k$ of continuous functions.
It is well-known that the $J_1$ topology relativized to $C^k$ coincides with the uniform topology on $C^k$, which is induced by the norm
\begin{equation*}
||x||_t := \sup_{0 \leq u \leq t} \|x(u)\|,
\end{equation*}
where $||x||$ denotes the usual Euclidean norm of $x \in \R^k$.
We use ``$\Ra$" to denote weak convergence of random variables in $\RR^k$, and of stochastic processes over compact time intervals.

For $f : \R^k \arr [0, \infty)$, $g : \R^k \arr [0, \infty)$ and $a \in \R_+^k \cup \{\infty\}$
we write $f(x) = O(g(x))$ as $x \arr a$ if $\limsup_{x \arr a} f(x)/g(x) < \infty$, and $f(x) = o(g(x))$ if  $\lim_{x \arr a} f(x)/g(x) = 0$.

For $x,y \in \RR$, we write $x \wedge y := \max\{x, y\}$.
For a function $x \in D$, $x(a-)$ denotes the left-hand limit at $a$, i.e., $x(a-) := \lim_{t \arr a^-} x(t)$ is the left-hand limit at the point $a$.

\subsection{Organization}
The rest of the paper is organized as follows. We introduce the model in Section \ref{sec:Model}.
In Section \ref{ap:pgf} we employ the buffer occupancy approach to establish a recursive characterization of the p.g.f.\ for the number in system at the polling epochs.
Based on that p.g.f., in Section \ref{ap:AllMoments} we develop sufficient condition (the existence of the m.g.f.'s for the service time and switchover time distributions) for the existence of all moments of the steady-state queue length.
In Section \ref{sec:FirstAndSecondMoment} we prove that the first- and second-order buffer occupancy equations each admits a unique solution, based on which the first and second moments can be calculated.
In addition, we relax the condition on the m.g.f.'s of the service time and switchover time distributions, and show that the second moment exists if and only if these distributions each possesses a finite variance.
In Section \ref{ap:AsympMoment} we apply the large-switchover-time scaling to the system, and characterize the limits for the moments of a sequence
of scaled stationary queue length.
We propose a simple approximation scheme for moments of all orders based on fluid models, which can be shown to be bona-fide fluid limits,
and demonstrate its effectiveness using simulation.
We conclude in Section \ref{sec:Conclusion}.

\section{The Model} \label{sec:Model}
We study a system of $K$ queues attended by a single server. Customers arrive at queue $k \in \mathcal K := \{1,...,K\}$ according to a Poisson process with rate $\lambda_{k} > 0$, and wait to be served in a infinite buffer.
The service times for customers at queue $k$ are independent and identically distributed (i.i.d) random variables with mean $1/\mu_k < \infty$. We use $S_k$ to denote a generic random variable that has the service time distribution of customers at queue $k$.

The server attends the queues periodically according to a polling table with $I$ stages, $I \geq K$.
We refer to each attendance of the server to the queue as a visit.
For $\mathcal I := \{1,...,I\}$, we define the polling function $p : \mathcal I \arr \mathcal K$ to map the stage to the corresponding queue visited by the server.
Note that a queue can be visited more than once in the polling table, and the polling table is said to be cyclic if each queue appears exactly once, in which case $I = K$.
Each transition of the server from stage $i$ to stage $i + 1 \, (\text{mod } I)$ incurs a random switchover time with mean $s_i $. We use $V_i$ to denote a generic random variable with the switchover time distribution from stage $i$ to stage $i + 1$. Let $s :=\sum_{i \in \mathcal I} s_i$ denote the total expected switchover time over the polling table, and assume that $s > 0$.

At each stage, the server serves the queue according to BEP \citep{levy1988optimization}, which is fully specified by a vector $\mathbf r = (r_1, ..., r_I) \in [0,1]^I$, whose component $r_i$ represents the average reduction proportion of the queue corresponding to stage $i \in \mathcal I$. If $\sum_{i \in \mathcal I : p(i) = k} r_i = 0$, then queue $k \in \mathcal K$ receives no service and is therefore unstable. Thus, we consider vector $\mathbf r$ to be in the set
\begin{equation*}
\mathcal R := \offf{\mathbf r \in [0,1]^I : \sum_{i \in \mathcal I : p(i) = k} r_i  > 0 \qforallq k \in \mathcal K} .
\end{equation*}

\begin{definition} [BEP]
	For $\mathbf r \in \mathcal R$, $i \in \mathcal I$, if the server finds $N$ customers at the polling epoch of stage $i$, then the server serves queue $p(i)$ until its queue length is reduced to level $N - Y_i(N, r_i)$, where $Y_i(N, r_i)$ is a binomial random variable with parameters $N$ and $r_i$, independently of all other random variables and processes.
\end{definition}

Equivalently, BEP can be considered as one in which the server draws an independent Bernoulli random variable with ``success probability" $r_i$ for each customer in the queue at the start of the stage $i$ visit, $i \in \mathcal I$. If the Bernoulli random variable is equal to one (i.e., a ``success"), then the server serves that customer together with all the newly arrived customers during its service time. Thus, the total time the server spends serving at stage $i$ is equal to the sum of $Y_i(N, r_i)$ busy periods of an $M/G/1$ queue with arrival rate $\lambda_{p(i)}$ and service time distributed as $S_{p(i)}$.

We refer to the starting time of a visit to a queue as a polling epoch, and the ending time of a visit as the departure epoch. A cycle is the time elapsed between two consecutive polling epochs of stage $1$.
For $m \geq 1$, $i \in \mathcal I$,
we let $A_i^{(m)}$ and $D_i^{(m)}$ denote the polling and departure epochs of stage $i$ in the $m$th cycle, respectively. Then $B_i^{(m)} := D_i^{(m)} - A_i^{(m)}$ is the busy time the server spends serving at stage $i$ in the $m$th cycle.
Without loss of generality, we take time $0$ to be the polling epoch of stage $1$ in the first cycle, i.e., $A_1^{(1)} := 0$.
We use $Q_k(t)$ to denote the number of customers in queue $k$ at time $t$, $k \in \mathcal K$, and write $Q(t) := (Q_k(t), k \in \mathcal K)$, $t \geq 0$.
Finally, we define $\tilde Q := \{Q(A_1^{(m)}) : m \geq 1\}$.

\paragraph{Stability.}
Let $\rho_k := \lambda_{k}/\mu_k$ denote the traffic intensity at queue $k$, and let $\rho := \sum_{k \in \mathcal K} \rho_k$,
and assume that $\rho < 1$, which is necessary and sufficient for the system to stabilize in a stationary distribution with integrable stationary queue and cycle lengths.
The proof of the following lemma follows from Proposition 1 and Theorem 3 in \cite{fricker1994monotonicity}.
Let $Q(A_1)$ be the random variable distributed according to the stationary distribution of $\tilde Q$,
so that $Q(A_1^{(m)}) \Arr Q(A_1)$ as $m \arr \infty$.
\begin{lemma} \label{lem:stability}
	$\tilde Q$ is an ergodic DTMC and $\E\off{Q(A_1)} < \infty$.
\end{lemma}

\section{Probability Generating Function of the Steady-State Queue Length} \label{ap:pgf}
In this section, we apply the buffer occupancy approach to derive the multi-dimensional p.g.f.\ of the stationary queue length embedded at the polling epochs.
In particular, we consider the system in steady state, operating under BEP with parameter $\mathbf r \in \mathcal R$. (The existence of the steady state follows from Lemma \ref{lem:stability}.)
Notation wise, we drop the transient time index $(m)$ to denote steady-state quantities. For example, for $i \in \mathcal I$, we use $A_i$ and $D_i$ to denote a generic polling and departure epoch of stage $i$ in steady state, respectively. $B_i := D_i - A_i$ is a generic random variable denoting the steady-state busy time the server spends serving at stage $i$.

Let $\hat R_i$, $i \in \mathcal{I}$, denote the Laplace-Stieltjes transform (LST) of the switchover time distribution when the server switches away from stage $i$.
In addition, let $\Theta_{p(i)}^{(\ell)}$, $i \in \mathcal{I}$, denote the busy period ``generated'' by the service of the $\ell$th served customer in queue $p(i)$ at the polling epoch of the queue,
and let $\hat \theta_k$, $k \in\mathcal{K}$, denote the LST of the busy period of an $M/G/1$ corresponding to queue $k$ that has arrival rate $\lm_k$
and service rate $\mu_k$.
All LSTs are defined with real (non-complex) exponents.

Let
\begin{equation*}
F_{i}(z_1,...,z_K) := \E\off{\prod_{k=1}^K z_k^{Q_k \of{A_i} }} , \quad i \in \mathcal{I},
\end{equation*}
denote the joint p.g.f.\ for $Q$ at time $A_i$.
Let the marginal p.g.f.\ for $Q_k$ at time $A_i$ be
\begin{equation} \label{eq:118}
G_{i,k}(u) := \E\off{u^{Q_k(A_i)}} = F_i(1,...,1, u, 1,...,1), \quad k \in \mathcal{K}, \quad i \in \mathcal{I},
\end{equation}
where $u$ is the $k$th argument for the joint p.g.f.\ $F_i$.
Let $\mathbf z := (z_1,...,z_K)$.

Given the p.g.f.'s, we can retrieve the moments of the steady-state queue length {\em whenever the moments exist}. In particular, let
\begin{equation*}
\begin{split}
&f_i(k) := \off{\frac{\partial F_i(z_1,...,z_K)}{\partial z_k} }_{\mathbf z = \mathbf 1}, \quad k \in \mathcal{K}, \quad i \in \mathcal{I} \\
&f_i(k,\ell) := \off{\frac{\partial^2 F_i(z_1,...,z_K)}{\partial z_k \partial z_\ell} }_{\mathbf z=\mathbf 1} , \quad k, \ell \in \mathcal{K}, \quad i \in \mathcal{I}.
\end{split}
\end{equation*}
Note that $F_i$ and $G_{i,k}$ are related via
\begin{equation*}
f_i(k) = G_{i,k}^{(1)}(1), \quad f_i(k,k) = G_{i,k}^{(2)}(1), \quad k \in \mathcal{K}, \quad i \in \mathcal{I},
\end{equation*}
where $G_{i,k}^{(1)}(1)$ and $G_{i,k}^{(2)}(1)$ are, respectively, the first and second derivative of $G_{i,k}$ at $u=1$.
Then
\begin{equation*}
\begin{split}
\E\off{Q_k(A_i)} &= f_i(k) , \quad k \in \mathcal{K}, \quad i \in \mathcal{I} \\
\E\off{Q_k(A_i) \of{Q_k(A_i) - 1}} &= f_i(k,k), \quad k \in \mathcal{K}, \quad i \in \mathcal{I}  \\
\E\off{Q_k(A_i) Q_j(A_i)} &= f_i(j,k), \quad j, k \in \mathcal{K}, \quad j \neq k, \quad i \in \mathcal{I} .
\end{split}
\end{equation*}

Let $w:\mathcal{I} \times \NN \arr \mathcal{I}$ be the function which retrieves the stage index in the polling table $j$
stages backward given that the server is currently at stage $i$, i.e.,
\begin{equation} \label{eq:97}
w(i,j) =
\begin{cases}
i - \of{ j \text{ mod } I }  \quad &\text{if }( j \text{ mod } I)  < i \\
I - \of{ j \text{ mod } I  -  i } \quad &\text{if } (j \text{ mod } I ) \geq i.
\end{cases}
\end{equation}
Correspondingly, $p(w(i,j))$ is the queue visited at stage $w(i,j)$.


\begin{proposition}  \label{lem:PGFBinExh}
	For all $i \in \mathcal{I}$ we have that
	\begin{equation} \label{eq:PGF}
	\begin{split}
	F_{i}\of{z_1, ...,z_K } &= \prod_{j=1}^\infty \hat R_{w(i,j)}\off{y^{(j-1)} },
	\end{split}
	\end{equation}
	where, with initial conditions $\mathbf z^{(0)} = \mathbf z$ and $y^{(0)} = \sum_{k=1}^K \of{\lambda_k - \lambda_k z_k}$,
	\begin{equation} \label{eq:PGF_recursion}
	\begin{split}
	z^{(j)}_\ell &=
	\begin{cases}
	z_\ell^{(j-1)} \quad &\text{for } \ell \neq p(w(i,j)) \\
	(1-r_{w(i,j)}) z_\ell^{(j-1)} + r_{w(i,j)} \hat \theta_\ell \off{y^{(j-1)}  - \of{\lambda_\ell - \lambda_\ell z_\ell^{(j-1)} }  } \quad &\text{for } \ell = p(w(i,j)) \\
	\end{cases} \\
	y^{(j)} &=  y^{(j-1)} + \lambda_{p(w(i,j))} \of{ z_{p(w(i,j))}^{(j-1)} - z_{p(w(i,j))}^{(j)}  } .
	\end{split}
	\end{equation}
	Furthermore,
	\begin{equation*}
	\begin{split}
	\lim_{m \arr \infty} \mathbf z^{(m)} = \mathbf 1, \quad
	\lim_{m \arr \infty} y^{(m)} = 0, \qandq
	\lim_{m \arr \infty} \hat R_{w(i,m)} \off{y^{(m-1)}} = 1, \quad i \in \mathcal{I} .
	\end{split}
	\end{equation*}
\end{proposition} 	
Note that $y^{(j)} := y^{(j)}(\mathbf z, i)$ for all $j \ge 0$, though we remove the dependence on $\mathbf z$ and $i$ from the expression to simplify the notation.
\begin{proof}{Proof.}
        In the paper thus far, we assumed (without loss of generality), that time $0$ is a polling epoch of stage $1$. In this proof,
        we drop this assumption as we need to analyze the transforms when initializing the queue process at a polling epoch for each stage.
	Let time $0$ be a polling epoch of some queue.
	Let $F_{i,m}$ denote the p.g.f.\ of the joint queue length at the polling epoch of stage $i$, $m$ stages after time $0$;
	in particular, we take the stage at time $0$ be $w(i,m)$, for $w$ defined in \eqref{eq:97}.
	In what follows, we derive $F_i$, the p.g.f.\ of the steady-state queue length at the polling epoch of stage $i$, by characterizing the
	p.g.f.\ $F_{i,m}$ and sending $m$ to infinity.
	
	Let $A_i(m)$ and $D_i(m)$ denote, respectively,
	the polling and departure epochs of stage $i$ after the server has completed $m$ visits since time $0$.
	We emphasize the distinction of $A_i(m)$ and $D_i(m)$ here from $A_i^{(m)}$ and $D_i^{(m)}$, which
	denote the polling and departure epochs of stage $i$ in the $m$th server cycle, respectively.	
	To relate $F_{i,m}$ to $F_{i+1,m+1}$, we first note that each of the class-$p(i)$ customers found at the polling epoch $A_i(m)$ independently ``generates"
	a busy period $\Theta_{p(i)}$ with probability $r_i$, and generates $0$ service duration with probability $1-r_i$. If a busy period is generated for a customer,
	then the joint p.g.f.\ for the number of arrivals at the other queues (other than queue $p(i)$) is given by
	\begin{equation*}
	\hat \theta_{p(i)} \off{\sum_{k=1, k\neq p(i)}^K \of{\lambda_k - \lambda_k z_k}  }.
	\end{equation*}
	If a customer is not selected to be served (i.e., a busy period is not generated), then it remains at the current queue and
	there are no corresponding arrivals since the service time of that customer is taken to be $0$. Hence,
	each of the customers found at the polling epoch of stage $i$ is replaced in an i.i.d.\ manner by a random population having p.g.f.\
	\begin{equation*}
	(1-r_i)z_{p(i)} + r_i \hat \theta_{p(i)}  \off{\sum_{k=1, k \neq p(i)}^K \lambda_k (1-z_k) } .
	\end{equation*}
	
	It follows that at time $ D_i(m)$, the joint p.g.f.\ for $Q(D_i(m))$ satisfies
	\begin{equation} \label{eq:102}
	\begin{split}
	&\E\off{\prod_{k=1}^K z_k^{Q_k \of{D_i(m)} }} \\
	= &\E\off{\of{(1-r_{i}) z_{p(i)} + r_{i} \hat \theta_{p(i)} \off{\sum_{k=1, k \neq p(i)}^K \of{\lambda_k - \lambda_k z_k}  }}^{Q_{p(i)}(A_i(m)) }  \prod_{k=1, k\neq p(i)}^K z_k^{Q_k(A_i(m)) }    } \\
	=& F_{i,m}\of{z_1, ...,z_{p(i)-1}, (1-r_{i}) z_{p(i)} + r_{i} \hat \theta_{p(i)} \off{\sum_{k=1, k \neq p(i)}^K \of{\lambda_k - \lambda_k z_k}  }, z_{p(i)+1},...,z_K}.
	\end{split}
	\end{equation}
	
	Now, the joint p.g.f.\ for the number of arrivals during the switchover time from stage $i$ to stage $i+1$ is given by
	\begin{equation} \label{eq:103}
	\hat R_{i}\off{\sum_{k=1}^K \of{\lambda_k - \lambda_k z_k}}.
	\end{equation}
	Since the events during the time interval $[A_i(m), D_i(m))$ are independent of all events in the time interval $[D_i(m), A_{i+1}(m+1)]$,
	the joint p.g.f.\ for $Q$ at time $A_{i+1}(m+1)$, i.e., the polling epoch of stage $i+1$,
	is given by the product of (\ref{eq:102}) and (\ref{eq:103}). Namely,  for $m \geq 1$,
	\begin{equation} \label{eq:104}
	\begin{split}
	&F_{(i+1 \text{ mod } I), m+1} \of{z_1, ...,z_K } = \hat R_{i}\off{\sum_{k=1}^K \of{\lambda_k - \lambda_k z_k}} \times \\
	&F_{i,m}\of{z_1,...,z_{p(i)-1}, (1-r_{i}) z_{p(i)} + r_{i} \hat \theta_{p(i)} \off{\sum_{k=1, k\neq p(i)}^K \of{\lambda_k - \lambda_k z_k}  }, z_{p(i)+1},...,z_K}.
	\end{split}
	\end{equation}

	Now, using (\ref{eq:104}) recursively $m$ times, we have
	\begin{equation} \label{eq:105}
	\begin{split}
	F_{i,m}\of{z_1, z_2,...,z_K } &= \prod_{j=1}^m \hat R_{w(i,j)}\off{\sum_{k=1}^K \of{\lambda_k - \lambda_k z_k^{(j-1)}}}
	F_{w(i,m), 0} (\mathbf z^{(m)})  ,
	\end{split}
	\end{equation}
	where  $\mathbf z^{(0)} =  \mathbf z$, and the elements of $\mathbf z^{(j)}$, for $j \ge 1$, are
	\begin{equation*} \label{eq:106}
	\begin{split}
	z^{(j)}_\ell &=
	\begin{cases}
	z_\ell^{(j-1)} \quad &\text{for } \ell \neq p(w(i,j)) \\
	(1-r_{w(i,j)}) z_\ell^{(j-1)}  + r_{w(i,j)} \hat \theta_\ell \off{\sum_{k=1, k\neq \ell}^K \of{\lambda_k - \lambda_k z_k^{(j-1)} }  } \quad &\text{for } \ell = p(w(i,j)). \\
	\end{cases}
	\end{split}
	\end{equation*}
	Note that $F_{w(i,m),0}$ in (\ref{eq:105}) is the p.g.f.\ of the initial queue length, namely,
	\begin{equation*}
	\begin{split}
	F_{w(i,m), 0}\of{z_1,...,z_K } &= \E\off{\prod_{k=1}^K z_k^{Q_k \of{0} }} .
	\end{split}
	\end{equation*}

	Letting $y^{(j)} := \sum_{k=1}^K \of{\lambda_k - \lambda_k z_k^{(j)}}$ for $j \geq 0$, retrieves \eqref{eq:PGF} and \eqref{eq:PGF_recursion}.
	The	p.g.f.\ $F_{i,m}$ converges to the p.g.f\ corresponding to the steady-state distribution of the queue as $m \tinf$ due to the stability of system under BEP. Since the steady-state limit is independent of the initial distribution,
	it must hold by \eqref{eq:105} that $F_{w(i,m), 0}\of{\mathbf z^{(m)}} \ra 1$, so that $\mathbf z^{(m)} \ra 1$ as $m\tinf$.
	As a result, $y^{(m)} \ra 0$ and $\hat R_{w(i,m)} \off{y^{(m-1)}} \ra 1$ as $m\tinf$. \Halmos
\end{proof}

\section{Sufficient Condition for Existence of All Moments} \label{ap:AllMoments}
We now prove that, if the service-time distribution at each queue has an m.g.f.\ that is finite in some neighborhood of zero, and if the switchover time distribution at each stage has finite m.g.f.\ on the positive real line,
then the steady-state queue length distribution at the polling epochs has finite moments of all orders.
Recall that $S_k$ is a generic random variable distributed according to the service time distribution at queue $k$,
and $A_i$ is the polling epoch of stage $i$ in stationarity, $k \in \mathcal{K}$, $i \in \mathcal{I}$.

\begin{theorem} \label{thm:AllMoments}
	Assume that the following two conditions hold:
	\begin{enumerate}[(a)]
		\item For each $k\in \mathcal{K}$, there exists $\epsilon_k >0$, such that $\E\off{e^{t S_k}} < \infty$ for all $t \in (-\epsilon_k, \epsilon_k)$.
		\item $\E\off{e^{t V_i}} < \infty$ for all $t \geq 0$ and $i \in \mathcal{I}$.
	\end{enumerate}
	Then $\E\off{Q \of{A_i}^\ell} < \infty$ for all $\ell \geq 1$ and $i \in \mathcal{I}$.
\end{theorem}
In the proof of Theorem \ref{thm:AllMoments} we will state auxiliary lemmas which are proved in Section \ref{ap:AuxAllMoments}.
However, before giving the proof of this theorem, we state the following lemma, whose proof also appears in Section \ref{ap:AuxAllMoments}.

\begin{lemma} \label{lem:PGFUpperBound}
	Under the conditions of Theorem \ref{thm:AllMoments}, it holds for $\{y^{(j)}: j \geq 0\}$ in \eqref{eq:PGF_recursion}, that
	\begin{equation*}
	\begin{split}
	F_{i}\of{z_1, ...,z_K } &\leq \prod_{\ell \in \mathcal{I}} \, \E\off{  e^{- \sum_{j = 1, \, w(i,j) = \ell}^\infty y^{(j-1)} V_\ell }  } \qforq i \in \mathcal I .
	\end{split}
	\end{equation*}
\end{lemma}

\begin{proof}{Proof of Theorem \ref{thm:AllMoments}.}
	Since $\E\off{e^{t V_\ell}} < \infty$ for all $t > 0$ and $\ell \in \mathcal{I}$ by assumption,
	the bound in Lemma \ref{lem:PGFUpperBound} is meaningful (finite) if $\sum_{j=0}^\infty y^{(j)} > -\infty$.
	We will show that the latter sum is finite for some $\mathbf z$ by proving that {\bf (i)} $y^{(j)} < 0$ for all $\mathbf z > \mathbf 1$ and
	for all $j \geq 0$, and that {\bf (ii)} there exists some $\mathbf z > \mathbf 1$ such that $\sum_{j=0}^\infty y^{(j)} > -\infty$.
	(Here and in what follows, $\mathbf z > \mathbf 1$ means that $z_k > 1$ for all $k \in \mathcal K$.)
	
	To establish {\bf (i)} and {\bf (ii)}, we first introduce an equivalent characterization of the $y^{(j)}$'s.
	Let $\mathbf y^{(j)} \in \R_+^K$ be the vector whose $k$th element is $ y_k^{(j)} := \lambda_k - \lambda_k z_k^{(j)}$.
	Note that $y^{(j)}$ is the sum of the elements in $\mathbf y^{(j)}$, namely, $y^{(j)} = \sum_{k =1}^K y_k^{(j)}$.
	The recursion calculating the $y^{(j)}$'s in \eqref{eq:PGF_recursion} can be equivalently written as
	\begin{equation} \label{eq:5}
	\begin{split}
	z^{(j)}_k &=
	\begin{cases}
	z_k^{(j-1)} \quad &\text{for } k \neq p(w(i,j)) \\
	(1-r_{w(i,j)}) z_k^{(j-1)} + r_{w(i,j)}  \hat \theta_k \off{ \sum_{\ell \in \mathcal{K}, \ell \neq k} y_\ell^{(j-1)} } \quad &\text{for } k = p(w(i,j)), \\
	\end{cases} \\
	y_{k}^{(j)} &= \lambda_{k} - \lambda_{k} z_{k}^{(j)} \\
	y^{(j)} &= \sum_{k =1}^K y_k^{(j)} ,
	\end{split}
	\end{equation}
	where $z_k^{(0)} =  z_k$ and $ y^{(0)}_k = \lambda_k - \lambda_k z_k$ for all $k \in \mathcal{K}$.
	It follows from \eqref{eq:5} that for $j \geq 1$,
	\begin{align*}
	y_k^{(j)} &= y_k^{(j-1)} \quad \text{for } k \neq p(w(i,j)) \\
	y_k^{(j)} &= \lambda_k - \lambda_k \of{(1-r_{w(i,j)}) z_k^{(j-1)} + r_{w(i,j)}  \hat \theta_k \off{ \sum_{\ell \in \mathcal{K}, \ell \neq k}  y_\ell^{(j-1)} } } \\
	&=  \lambda_k - \lambda_k  z_k^{(j-1)} + \lambda_k r_{w(i,j)} z_k^{(j-1)} - \lambda_k r_{w(i,j)}  \hat \theta_k \off{ \sum_{\ell \in \mathcal{K}, \ell \neq k}  y_\ell^{(j-1)} } \\
	&=  \lambda_k - \lambda_k  z_k^{(j-1)} + \lambda_k r_{w(i,j)} z_k^{(j-1)} - \lambda_k r_{w(i,j)}   - \lambda_k r_{w(i,j)}  \of{ \hat \theta_k \off{ \sum_{\ell \in \mathcal{K}, \ell \neq k}  y_\ell^{(j-1)} } -1 } \\
	&= (1-r_{w(i,j)}) (\lambda_k - \lambda_k  z_k^{(j-1)} ) - \lambda_k r_{w(i,j)}  \of{ \hat \theta_k \off{ \sum_{\ell \in \mathcal{K}, \ell \neq k} y_\ell^{(j-1)} } -1 } \\
	&= (1-r_{w(i,j)}) y_k^{(j-1)} - \lambda_k r_{w(i,j)}  \of{ \hat \theta_k \off{ \sum_{\ell \in \mathcal{K}, \ell \neq k}  y_\ell^{(j-1)} } -1 } \quad \text{for } k = p(w(i,j)) .
	\stepcounter{equation}\tag{\theequation}\label{eq:6}
	\end{align*}
	
	\paragraph{Proof of (i).}
Let $\mathbf z > \mathbf 1$.
	To see $y^{(j)} < 0$ for all $j \geq 0$,
	note that $ y_k^{(0)} = \lambda_k - \lambda_k z_k^{(0)} < 0$ for all $k \in \mathcal{K}$.
	Then, $\hat \theta_k \off{ \sum_{\ell \in \mathcal{K}, \ell \neq k} y_\ell^{(0)} } > 1$ for $k = p(w(i,1))$.
	By \eqref{eq:6},
	\begin{equation*}
	\begin{split}
	y_k^{(1)} &= y_k^{(0)} < 0 \quad \text{for } k \neq p(w(i,1)) \\
	y_k^{(1)} &= (1-r_{w(i,1)})  y_k^{(0)} - \lambda_k r_{w(i,1)}  \of{ \hat \theta_k \off{ \sum_{\ell \in \mathcal{K}, \ell \neq k} y_\ell^{(0)} } -1 } < 0 \quad \text{for } k = p(w(i,1)) .
	\end{split}
	\end{equation*}
	We can then use the arguments inductively to show that $y_k^{(j)} < 0$ for all $k \in \mathcal{K}$ and $j \geq 2$.
	Thus, $y^{(j)} = \sum_{k \in \mathcal{K}} y_k^{(j)} < 0$ for all $j \geq 0$.
	
	
	\paragraph{Proof of (ii).} 

	We take the following steps to prove claim (ii):
	We first transform \eqref{eq:6} into an affine system of equations (see \eqref{eq:10}) which is more amenable to analysis. We show that there exists some $\mathbf z > \mathbf 1$ under which \eqref{eq:10} yields a $K$-dimensional vector $\mathbf{\bar y}^{(j)} := (\bar  y^{(j)}_k, k \in \mathcal{K})$ with $\bar y_k^{(j)} \leq y_k^{(j)} < 0$ for all $k\in \mathcal{K}, j \geq 0$ (see Lemma \ref{lem:Y_Bound}).  Let $\bar y^{(j)} := \sum_{k \in \mathcal{K} } \bar  y^{(j)}_k$. We then show that $\sum_{j=0}^\infty \bar y^{(j)} > -\infty$.
	In particular, we consider the vectors $\mathbf{\tilde y}^{(j)} := (\tilde y^{(j)}_k, k \in \mathcal{K})$ with $\tilde y_k^{(j)} := \bar  y_k^{(j)} / \rho_k$ for all $k \in \mathcal{K}$ and $j \geq 0$. This transformation allows us to derive an appropriate contraction map (see Lemma \ref{lem:102}).We next make the above outline rigorous.
	
	
	
	
	
	By Proposition 4.2 in \cite{nakayama2004finite},
	there exists $\zeta_k >0$ such that $\E\off{e^{t \Theta_k}} < \infty$ for all $t \in [-\zeta_k, \zeta_k]$.
	For $k \in \mathcal{K}$ and $t \in [-\zeta_k/2, \zeta_k/2]$,
	\begin{equation} \label{eq:7}
	\begin{split}
	e^{t \Theta_k} = 1 + \Theta_k t + \frac{1}{2}\Theta_k^2 \, e^{\xi \Theta_k} \, t^2 , \quad w.p.1,
	\end{split}
	\end{equation}
	for some $\xi \in [0,t]$ by Taylor's expansion, where the term $\frac{1}{2}\Theta_k^2 e^{\xi \Theta_k} t^2$ is the Lagrange form of the remainder.
	Since $\xi \leq t \leq \zeta_k/2$, we have
	\begin{equation*}
	e^{t \Theta_k}  \leq 1 + \Theta_k t + \frac{1}{2}\Theta_k^2 e^{\frac{1}{2} \zeta_k \Theta_k} t^2 .
	\end{equation*}
	Taking expectation on both sides of \eqref{eq:7} and applying H\"{o}lder's inequality gives
	\begin{equation} \label{eq:17}
	\begin{split}
	\E\off{ e^{t \Theta_k} }
	&\leq 1 + \E\off{\Theta_k} t + \frac{1}{2} \E\off{ \Theta_k^4}^{\frac{1}{2}} \E\off{e^{\zeta_k \Theta_k} }^{\frac{1}{2}} t^2 .
	\end{split}
	\end{equation}
	
	Take $M_k := \E\off{ \Theta_k^4}^{\frac{1}{2}} \E\off{e^{\zeta_k \Theta_k} }^{\frac{1}{2}}$, and note that $M_k < \infty$ by the choice of $\zeta_k$.
	Let $\alpha > 0$ be such that
	\begin{equation} \label{eq:8}
	\rho_{k} + \sum_{\ell \in \mathcal{K}, \ell \neq k} \rho_\ell (1+\alpha) < 1 \quad \text{for all } k \in \mathcal{K} .
	\end{equation}
	The existence of such $\alpha$ is guaranteed by the fact that $\rho < 1$.
	
	Let
	$$\zeta := \min_{k \in \mathcal{K}} \offf{ \zeta_k/2 \, \wedge \, \frac{2\E\off{\Theta_k}\alpha}{M_k} }.$$
	Now for $t \in [-\zeta, \zeta]$, it holds that
	\begin{equation} \label{eq:18}
	\frac{1}{2} \E\off{ \Theta_k^4}^{\frac{1}{2}} \E\off{e^{\zeta_k \Theta_k} }^{\frac{1}{2}} t^2 = \frac{1}{2} M_k t^2
	\leq \frac{1}{2} M_k  \frac{2\E\off{\Theta_k}\alpha}{M_k} |t| = \E\off{\Theta_k} \alpha |t| ,  \quad k \in \mathcal{K}.
	\end{equation}
	Plugging \eqref{eq:18} into \eqref{eq:17}, we get that for $t \in [-\zeta, \zeta]$,
	\begin{equation} \label{eq:9}
	\E\off{ e^{t \Theta_k} } \leq 1 + \E\off{\Theta_k} t +  \E\off{\Theta_k} \alpha |t|,  \quad k \in \mathcal{K} .
	\end{equation}
	Applying \eqref{eq:9} to \eqref{eq:6}, we see that for $k \in \mathcal{K}$, if $-\zeta \leq \sum_{\ell \in \mathcal{K}, \ell \neq k} y_\ell^{(j-1)} < 0$,
	then
	\begin{equation} \label{eq:27}
	\begin{split}
	\hat \theta_k \off{ \sum_{\ell \in \mathcal{K}, \ell \neq k}  y_\ell^{(j-1)} }
	&= \E\off{e^{-\Theta_k \sum_{\ell \in \mathcal{K}, \ell \neq k}  y_\ell^{(j-1)}}} \\
	&\leq 1 - \E\off{\Theta_k} \sum_{\ell \in \mathcal{K}, \ell \neq k} y_\ell^{(j-1)} - \E\off{\Theta_k} \alpha \sum_{\ell \in \mathcal{K}, \ell \neq k} y_\ell^{(j-1)} \\
	&= 1 - \E\off{\Theta_k} (1+\alpha) \of{  \sum_{\ell \in \mathcal{K}, \ell \neq k} y_\ell^{(j-1)}   } .
	\end{split}
	\end{equation}
	
	Based on the recursion in \eqref{eq:6} and motivated by the upper bound constructed in \eqref{eq:27}, we consider the following recursion
	\begin{equation}  \label{eq:10}
	\begin{split}
	{\bar y}_k^{(j)} &= {\bar y}_k^{(j-1)} \quad \text{for } k \neq p(w(i,j)) \\
	{\bar y}_k^{(j)} &= (1-r_{w(i,j)})   {\bar y}_k^{(j-1)} +  \lambda_k r_{w(i,j)} \E\off{\Theta_k} (1+\alpha) \of{  \sum_{\ell \in \mathcal{K}, \ell \neq k}  {\bar y}_\ell^{(j-1)}   } \\
	&= (1-r_{w(i,j)}) {\bar y}_k^{(j-1)} +  \lambda_k r_{w(i,j)} \frac{1}{\mu_k - \lambda_k} (1+\alpha) \of{  \sum_{\ell \in \mathcal{K}, \ell \neq k} {\bar y}_\ell^{(j-1)}   }   \\
	&= (1-r_{w(i,j)})  {\bar y}_k^{(j-1)} +  r_{w(i,j)} \frac{\rho_k}{1 - \rho_k} (1+\alpha) \of{  \sum_{\ell \in \mathcal{K}, \ell \neq k}  {\bar y}_\ell^{(j-1)}   }
	\qforq k = p(w(i,j)) ,
	\end{split}
	\end{equation}
	where	$ {\bar y}_k^{(0)} =  y_k^{(0)} = \lambda_k - \lambda_k  z_k$, for $k \in \mathcal{K}$.

	For $\ell \in \mathcal{I}$, let $\bar{\mathcal M}_\ell$ denote the $K \times K$ matrix whose $p(\ell)$th row has $1-r_{p(\ell)}$ for the $p(\ell)$th entry,
	and $r_{p(\ell)} \rho_{p(\ell)}(1+\alpha)/(1-\rho_{p(\ell)})$ in all the other entries of this row;
	for $k \neq p(\ell)$, the $k$th row is the vector $\mathbf u_k$ that has $1$ in the $k$th entry and $0$ elsewhere;
	see Example \ref{ex:MatrixM} below for a concrete representation of this matrix for a system with a cyclic polling table with $K=3$.
	
	Let $\mathbf{\bar y}^{(j)} := ({\bar y}_k^{(j)}, k \in \mathcal K)$.
	It follows from	\eqref{eq:10} that $\mathbf{\bar y}^{(j)}$ can be derived via sequential multiplication of matrices $\bar{\mathcal{M}}_\ell$, $\ell \in \mathcal{I}$.
	In particular,
	define the $K \times K$ matrix
	$$\bar{\mathcal M}(i) := \bar{\mathcal M}_{w(i,I)} \bar{\mathcal M}_{w(i,I-1)} \dotsm \bar{\mathcal M}_{w(i,1)}.$$
	It holds that
	\begin{equation*}
	\begin{split}
	\mathbf{\bar y}^{(j)}
	&= \bar{\mathcal{M}}_{w(i,j)} \, \bar{\mathcal{M}}_{w(i,j-1)} \, \dotsm \bar{\mathcal{M}}_{w(i,1)} \, \mathbf{\bar y}^{(0)} \\
	&= \bar{\mathcal{M}}_{w(i,j \text{ mod } I)} \, \bar{\mathcal{M}}_{w(i,(j \text{ mod } I)-1)}  \dotsm \bar{\mathcal{M}}(i)^{\lfloor j/I \rfloor} \, \mathbf{\bar y}^{(0)},
	\end{split}
	\end{equation*}
	where, for $x \in \RR$, $\lfloor x \rfloor$ is the largest integer that is smaller than $x$ (the floor function).
	
	Next, let $\mathbf{\tilde y}^{(j)} := ({\tilde y}^{(j)}_k, k \in \mathcal{K})$, where
	\begin{equation} \label{eq:tilde y}
	{\tilde y}^{(j)}_k := {\bar y}^{(j)}_k / \rho_k, \quad k \in \mathcal{K}, \,\,  j \geq 0.
	\end{equation}
	Let $\tilde{\mathcal{M}}_\ell$ be the $K \times K$ matrix whose $p(\ell)$th row has $1-r_{p(\ell)}$ for the $p(\ell)$th entry, and $r_{p(\ell)} \rho_k(1+\alpha)/(1-\rho_{p(\ell)})$
	for entry $k \neq p(\ell)$, $k \in \mathcal{K}$; for row $k \neq p(\ell)$, it is equal to $\mathbf u_k$.
	(The construction of the $\tilde{\mathcal{M}}_\ell$ matrices is also illustrated in Example \ref{ex:MatrixM}.)
	Similar to $\bar{\mathcal{M}}(i)$, define $\tilde{\mathcal M}(i) := \tilde{\mathcal M}_{w(i,I)} \tilde{\mathcal M}_{w(i,I-1)} \dotsm \tilde{\mathcal M}_{w(i,1)}$.
	By construction,
	\begin{equation} \label{eq:101}
	\begin{split}
	\mathbf{\tilde y}^{(j)} &= \tilde{\mathcal{M}}_{w(i,j)} \, \tilde{\mathcal{M}}_{w(i,j-1)} \dotsm \tilde{\mathcal{M}}_{w(i,1)} \, \mathbf{\tilde y}^{(0)} \\
	&= \tilde{\mathcal{M}}_{w(i,j \text{ mod } I)} \, \tilde{\mathcal{M}}_{w(i,(j \text{ mod } I)-1)}  \dotsm \tilde{\mathcal{M}}(i)^{\lfloor j/I \rfloor}
	\, \mathbf{\tilde y}^{(0)} .
	\end{split}
	\end{equation}

	For $k \in \mathcal K$ define
	\begin{equation*}
	d_k(i) := \min \offf{j \geq 1: p(w(i,j)) = k, \, r_{w(i,j)} > 0},
	\end{equation*}
	namely, $w(i,d_k(i))$ is the last stage in the polling table, prior to stage $i$, that the service ratio of queue $k$ is strictly positive.
	(When all the service ratios of all stages are strictly positive, $w(i,d_k(i))$ is simply the last visit to queue $k$ prior to stage $i$.)
	Since, for BEP with parameter $\mathbf r \in \mathcal R$, it holds that $d_k(i) \leq I$.
	Let $(\mathcal A)_{\sum_k}$ denote the sum of the entries in the $k$th row of the matrix $\mathcal A$.

	\begin{lemma} \label{lem:102}
		For each $k \in \mathcal{K}$,
		it holds that
		\begin{equation*}
		\of{  \tilde{\mathcal M}_{w(i,\ell)} \,  \tilde{\mathcal M}_{w(i,\ell-1)}  \dotsm \tilde{\mathcal M}_{w(i,1)}  }_{\sum k} \leq  1, \quad \text{for all } \ell \geq 1.
		\end{equation*}
		Furthermore,
		\begin{equation*}
		\of{  \tilde{\mathcal M}_{w(i,\ell)}  \, \tilde{\mathcal M}_{w(i,\ell-1)}  \dotsm \tilde{\mathcal M}_{w(i,1)}  }_{\sum k} <  1 , \quad \text{for all } \ell \geq d_k(i).
		\end{equation*}
	\end{lemma}
	
	Now, select $\mathbf z > \mathbf 1$ (sufficiently close to $\mathbf 1$) such that
	\begin{equation} \label{eq:z}
	{\tilde y}_k^{(0)} = {\bar y}_k^{(0)} / \rho_k = (\lambda_k - \lambda_k z_k) / \rho_k \geq -\zeta \rho_k / K, \quad \text{for each } k \in \mathcal{K} .
	\end{equation}
	By the characterization of ${\tilde y}_k^{(j)}$ in \eqref{eq:101} and Lemma \ref{lem:102}, together with the fact that each element in the matrices $\tilde{\mathcal{M}}_\ell$, $\ell \in \mathcal{I}$, is non-negative,  it holds that ${\tilde y}_k^{(j)} \in [-\zeta\rho_k/K, 0)$ for all $k \in \mathcal{K}$ and $j \geq 0$. Thus, ${\bar y}_k^{(j)} \in [-\zeta/K, 0)$ for all $k \in \mathcal{K}$ and $j \geq 0$ by \eqref{eq:tilde y}. It follows that
	\begin{equation} \label{eq:zSuchThat}
	\bar y^{(j)}:=\sum_{k \in \mathcal{K}} {\bar y}_k^{(j)} \in [-\zeta, 0),  \quad \text{for all } j \geq 0 .
	\end{equation}
	
	The next lemma shows that for the choice of $\mathbf z$ in \eqref{eq:z}, the ${\bar y}_k^{(j)}$'s calculated according to the new recursion in \eqref{eq:10} is a lower bound for the $ y_k^{(j)}$'s calculated according to the original recursion in \eqref{eq:6}.

	\begin{lemma} \label{lem:Y_Bound}
		For  $\mathbf z > \mathbf 1$ such that \eqref{eq:z} holds,
		we have ${\bar y}_k^{(j)} \leq y_k^{(j)} < 0$, for all $k \in \mathcal{K}$ and $j \geq 0$.
	\end{lemma}

	By Lemma \ref{lem:Y_Bound},  claim {\bf (ii)} follows if we show that $\sum_{j = 0}^{\infty} {\bar y}^{(j)} > -\infty$, which is equivalent to $\sum_{j = 0}^{\infty} {\tilde y}^{(j)} > -\infty$ by \eqref{eq:tilde y}.
	To this end,
	we group elements in the sequence $\{\mathbf{\tilde y}^{(j)} : j \geq 0\}$ into $I$ subsequences. For $\ell \in \mathcal{I}$, the $\ell$th subsequence is defined as $\{\mathbf{\tilde y}^{(j)} : (j \text{ mod } I) = \ell-1, j \geq 0\}$.
	Without loss of generality, consider the first subsequence  $\{\mathbf{\tilde y}^{(j)} : (j \text{ mod } I) = 0, j \geq 0\}$. By \eqref{eq:101}, this subsequence can be equivalently represented
	as
	\begin{equation*}
	\{\mathbf{\tilde y}^{(j)} : (j \text{ mod } I) = 0 , j \geq 0\} = \{ \tilde{\mathcal M}(i)^n \, \mathbf{\tilde y}^{(0)} : n \geq 0 \} .
	\end{equation*}
	We will show that the partial sum of this subsequence converges, namely,
	\begin{equation} \label{eq:12}
	\lim_{n \arr \infty} \, \sum_{j = 0, \, (j \text{ mod } I) = 0}^n 	\mathbf{\tilde y}^{(j)} =
	\lim_{n \arr \infty} \, \sum_{j = 0, \, (j \text{ mod } I) = 0}^n  \tilde{\mathcal M}(i)^n \, \mathbf{\tilde y}^{(0)} > - \infty .
	\end{equation}
	Note that \eqref{eq:12} holds if and only if $\varrho(\tilde{\mathcal{M}}(i)) < 1$,
	where $\varrho(\tilde{\mathcal{M}}(i))$ is the spectral radius of $\tilde{\mathcal{M}}(i)$.
	To this end, consider the matrix norm $||\cdot||_\infty$ defined as follows
	\begin{equation*}
	||\mathcal{\tilde{\mathcal{M}} }(i) ||_\infty := \max_{k \in \mathcal{K}} (\mathcal{\tilde{\mathcal{M}}}(i))_{\sum k} .
	\end{equation*}
	By Lemma \ref{lem:102} and the fact that $d_k(i) \leq I$,
	it holds that $||\tilde{\mathcal{M}}(i)||_\infty < 1$, so that $\varrho(\tilde{\mathcal{M}}(i)) \leq ||\tilde{\mathcal{M}}(i)||_\infty < 1$.
	Hence, we have the convergence result in \eqref{eq:12}.
	The same lines of reasoning can be applied to the rest $I -1$ subsequences to show that
	\begin{equation*}
	\sum_{\ell \in \mathcal{I}} \, \lim_{n \arr \infty} \, \sum_{j = 0, \, (j \text{ mod } I) = \ell-1}^n 	\mathbf{\tilde y}^{(j)} > - \infty ,
	\end{equation*}
	which in turn implies that $\sum_{j = 0}^{\infty} \bar y^{(j)} > -\infty$, and $\sum_{j = 0}^{\infty} y^{(j)} > -\infty$. \Halmos
\end{proof}

\subsection{Proofs of the Auxiliary Results in the Proof of Theorem \ref{thm:AllMoments}} \label{ap:AuxAllMoments}
\begin{proof} {Proof of Lemma \ref{lem:PGFUpperBound}.}
	Let $\ell \in \mathcal{I}$. The lemma follows if we show that
	\begin{equation*} \label{eq:117}
	\begin{split}
	\prod_{j = 1, \, w(i,j) = \ell}^\infty \E\off{  e^{- y^{(j-1)} V_\ell }  }  \leq \E\off{  e^{- \sum_{j = 1, \, w(i,j) = \ell}^\infty y^{(j-1)} V_\ell }  } .
	\end{split}
	\end{equation*}
	To this end,
	note from \eqref{eq:6} and the proof for claim {\bf (i)} that for any $\mathbf z > \mathbf 1$, we have $\mathbf y^{(j)} < \mathbf 0$ for all $j \geq 0$.
	By the assumption that $\E\off{e^{tV_\ell}} < \infty$ for all $t > 0$ and \cite[p.2]{schmidt2003covariance}, we get
	\begin{equation*}
	\E\off{  e^{- y^{(0)} V_\ell }  } \E\off{  e^{- y^{(1)} V_\ell }  } \leq \E\off{ e^{- y^{(0)} V_\ell } e^{- y^{(1)} V_\ell } }.
	\end{equation*}
	Let $n \geq 1$. For the induction hypothesis, suppose that
	\begin{equation} \label{eq:116}
	\begin{split}
	\prod_{j = 1, \, w(i,j) = \ell}^n \E\off{  e^{- y^{(j-1)} V_\ell }  }  \leq \E\off{  e^{- \sum_{j = 1, \, w(i,j) = \ell}^n y^{(j-1)} V_\ell }  } .
	\end{split}
	\end{equation}
	Then,
	\begin{equation*}
	\begin{split}
	\prod_{j = 1, \, w(i,j) = \ell}^{n+1} \E\off{  e^{- y^{(j-1)} V_\ell }  }
	&= \prod_{j = 1, \, w(i,j) = \ell}^{n} \E\off{  e^{- y^{(j-1)} V_\ell }  } \E\off{  e^{- y^{(n)} V_\ell }  } \\
	&\leq \E\off{  e^{- \sum_{j = 1, \, w(i,j) = \ell}^n y^{(j-1)} V_\ell }  } \E\off{  e^{- y^{(n)} V_\ell }  }  \\
	&\leq  \E\off{  e^{- \sum_{j = 1, \, w(i,j) = \ell}^{n+1} y^{(j-1)} V_\ell }  } ,
	\end{split}
	\end{equation*}
	where the first inequality follows from \eqref{eq:116}, and the second inequality follows from \cite[p.2]{schmidt2003covariance}, together with the fact (implied again by the assumption $\E\off{e^{tV_\ell}} < \infty$ for all $t > 0$) that
	\begin{equation*}
	\E\off{  \of{e^{- \sum_{j = 1, \, w(i,j) = \ell}^n y^{(j-1)} V_\ell }}^2  } < \infty \quad \text{and} \quad \E\off{  \of{e^{- y^{(n)} V_\ell }}^2  } < \infty .
	\end{equation*}
	We have shown by induction that
	\begin{equation*}
	\begin{split}
	\prod_{j = 1, \, w(i,j) = \ell}^n \E\off{  e^{- y^{(j-1)} V_\ell }  }  \leq \E\off{  e^{- \sum_{j = 1, \, w(i,j) = \ell}^n y^{(j-1)} V_\ell }  } \quad \text{for all } n \geq 1.
	\end{split}
	\end{equation*}
	Sending $n$ to infinity gives
	\begin{equation*}
	\begin{split}
	\lim_{n \arr \infty} \prod_{j = 1, \, w(i,j) = \ell}^n \E\off{  e^{- y^{(j-1)} V_\ell }  }  &\leq \lim_{n \arr \infty} \E\off{  e^{- \sum_{j = 1, \, w(i,j) = \ell}^n y^{(j-1)} V_\ell }  } \\
	&=  \E\off{ \lim_{n \arr \infty} e^{- \sum_{j = 1, \, w(i,j) = \ell}^n y^{(j-1)} V_\ell }  },
	\end{split}
	\end{equation*}
	where the interchange of limit and expectation in the second line is justified by the monotone convergence theorem. \Halmos
\end{proof}

\begin{proof} {Proof of Lemma \ref{lem:Y_Bound}.}
	Note from \eqref{eq:6} and the proof for claim {\bf (i)} that for $\mathbf z > \mathbf 1$, we have $\mathbf y^{(j)} < \mathbf 0$ for all $j \geq 0$.
	First, it follows by construction that $\bar y_k^{(0)} = y_k^{(0)} < 0$ for all $k \in \mathcal{K}$.
	
	Then, let
	$k = p(w(i,1))$. Given $\mathbf{ y}^{(0)}$, $\mathbf{ y}^{(1)}$ is calculated as
	\begin{equation*}
	\begin{split}
	{y}_\ell^{(1)} &=  {y}_\ell^{(0)} \quad \text{for } \ell \neq k \\
	{y}_k^{(1)}
	&= (1-r_{w(i,1)})  {y}_k^{(0)} - \lambda_k r_{w(i,1)}  \of{ \hat \theta_k \off{ \sum_{\ell \in \mathcal{K}, \ell \neq k} { y}_\ell^{(0)} } -1 } .
	\end{split}
	\end{equation*}
	Since \eqref{eq:zSuchThat} holds by the choice of $\mathbf z$ in \eqref{eq:z}, and since $\bar{\mathbf y}^{(0)} < \mathbf 0$, we have
	\begin{equation*}
	-\zeta < \sum_{\ell \in \mathcal{K}, \ell \neq k}  {\bar y}_\ell^{(0)} = \sum_{\ell \in \mathcal{K}, \ell \neq k}  {y}_\ell^{(0)}  < 0 ,
	\end{equation*}
	so that $\hat \theta_k \off{ \sum_{\ell \in \mathcal{K}, \ell \neq k}  y_\ell^{(0)} } < \infty$.
	Using the upper bound in \eqref{eq:27}, we get
	\begin{equation} \label{eq:13}
	\begin{split}
	y_k^{(1)}
	&\geq (1-r_{w(i,1)})  y_k^{(0)} - \lambda_k r_{w(i,1)}  \of{ -	\E\off{\Theta_k} (1+\alpha) \sum_{\ell \in \mathcal{K}, \ell \neq k}  y_\ell^{(0)}  } \\
	&= (1-r_{w(i,1)})  y_k^{(0)} +  r_{w(i,1)} \frac{\rho_k}{1 - \rho_k} (1+\alpha) \of{  \sum_{\ell \in \mathcal{K}, \ell \neq k}  y_\ell^{(0)}   }  \\
	&= {\bar y}_k^{(1)} .
	\end{split}
	\end{equation}
	For $\ell \neq k$, we have $ y_\ell^{(1)} = y_\ell^{(0)} = {\bar y}_\ell^{(0)} = {\bar y}_\ell^{(1)}$. Thus, $ {\bar y}_k^{(1)} \leq y_k^{(1)} < 0$ for all $k \in \mathcal{K}$.
	
	Now, let $k = p(w(i,2))$. By \eqref{eq:zSuchThat} (implied by the choice of $\mathbf z$ in \eqref{eq:z}) and the fact that $\bar{\mathbf y}^{(1)} \leq \mathbf y^{(1)} < \mathbf 0$, we have
	\begin{equation*}
	-\zeta < \sum_{\ell \in \mathcal{K}, \ell \neq k}  {\bar y}_\ell^{(1)} = \sum_{\ell \in \mathcal{K}, \ell \neq k}  {y}_\ell^{(1)}  < 0 ,
	\end{equation*}
	so that $\hat \theta_{k} \off{ \sum_{\ell \in \mathcal{K}, \ell \neq p(w(i,2)) } y_\ell^{(1)} } < \infty$. Similarly to \eqref{eq:13}, we have
	\begin{equation*}
	\begin{split}
	y_\ell^{(2)} &=  y_\ell^{(1)} \quad \text{for } \ell \neq k \\
	y_k^{(2)}
	&\geq (1-r_{w(i,2)}) y_k^{(1)} +  r_{w(i,2)} \frac{\rho_k}{1 - \rho_k} (1+\alpha) \of{  \sum_{\ell \in \mathcal{K}, \ell \neq k}  y_\ell^{(1)}   }
	= {\bar y}_k^{(2)} .
	\end{split}
	\end{equation*}
	For $\ell \neq k$, we have $ y_\ell^{(2)} = y_\ell^{(1)} \geq  {\bar y}_\ell^{(1)} = {\bar y}_\ell^{(2)}$. Thus, $ {\bar y}_k^{(2)} \leq y_k^{(2)} < 0$ for all $k \in \mathcal{K}$.
	The same argument can be applied inductively, and the claim follows. \Halmos
\end{proof}

\begin{proof} {Proof of Lemma \ref{lem:102}.}
	By construction,
	\begin{equation*} \label{eq:110}
	\begin{split}
	\of{\tilde{\mathcal{M}}_{w(i,1)}}_{\sum p(w(i,1)) } &= 1-r_{w(i,1)} + \sum_{k \in \mathcal{K}, \, k \neq p(w(i,1))} r_{w(i,1)} \frac{\rho_k(1+\alpha)}{1-\rho_{p(w(i,1))}}  \\
	&= 1- r_{w(i,1)}  \frac{1-\rho_{p(w(i,1))} - \sum_{k \in \mathcal{K}, k \neq p(w(i,1))} \rho_k(1+\alpha) }{1-\rho_{p(w(i,1))}} \\
	&\leq 1 ,
	\end{split}
	\end{equation*}
	where the inequality follows from the choice of $\alpha$ in \eqref{eq:8}, and holds strictly if $r(w(i,1)) > 0$.
	In addition, $\of{\tilde{\mathcal{M}}_{w(i,1)}}_{\sum k} = 1$, for $k \neq p(w(i,1))$, $k \in \mathcal{K}$.

	Now, consider the matrix product $\tilde{\mathcal{M}}_p :=  \tilde{\mathcal{M}}_{w(i,2)} \tilde{\mathcal{M}}_{w(i,1)}$.
	For $k \neq  p(w(i,2)), k \in \mathcal{K}$,
	the $k$th row of $\tilde{\mathcal{M}}_{w(i,2)} $ is equal to $\mathbf u_k$ (the unit vector). Thus, the $k$th row of $\tilde{\mathcal{M}}_p$ is identical to that of $\tilde{\mathcal{M}}_{w(i,1)} $, for all $k \neq  p(w(i,2)), k \in \mathcal{K}$.
	Now, for the $p(w(i,2))$th row, we have
	\begin{equation*} \label{eq:100}
	\begin{split}
	&\of{\tilde{\mathcal{M}}_p}_{\sum p(w(i,2)) } \\
	&=
	(1-r_{w(i,2)}) \of{\tilde{\mathcal{M}}_{w(i,1)}}_{\sum p(w(i,2)) } + \sum_{k \in \mathcal{K}, \, k \neq p(w(i,2)) } r_{w(i,2)} \frac{\rho_k(1+\alpha)}{1-\rho_{p(w(i,2))}} \of{\tilde{\mathcal{M}}_{w(i,1)}}_{\sum k } \\
	&\leq 1-r_{w(i,2)} + \sum_{k \in \mathcal{K}, \, k \neq p(w(i,2)) } r_{w(i,2)} \frac{\rho_k(1+\alpha)}{1-\rho_{p(w(i,2))}} \\
	&= 1- r_{w(i,2)} \frac{1-\rho_{p(w(i,2))} - \sum_{k \in \mathcal{K}, k \neq p(w(i,2)) } \rho_k(1+\alpha) }{1-\rho_{p(w(i,2))}} \\
	&\leq 1 ,
	\end{split}
	\end{equation*}
	where the first inequality follows from
	$\of{\tilde{\mathcal{M}}_{w(i,1)}}_{\sum k} \leq 1$ for all $k \in \mathcal{K}$,
	and the second inequality again follows form the choice of $\alpha$ in \eqref{eq:8}.

	We can extend the arguments by considering multiplying matrix $\tilde{\mathcal{M}}_{w(i,3)}$ to $\tilde{\mathcal{M}}_{p}$, and the proof continues inductively. Importantly, for $k \in\mathcal{K}$, the $k$th row sum of the resulting matrix becomes strictly less than 1 when multiplying $\tilde{\mathcal{M}}_{w(i,d_k(i))}$ to $\tilde{\mathcal M}_{w(i,d_k(i)-1)}   \tilde{\mathcal M}_{w(i,d_k(i)-2)}  \dotsm \tilde{\mathcal M}_{w(i,1)} $. Moreover, the $k$th row sum of $  \tilde{\mathcal M}_{w(i,\ell)} \tilde{\mathcal M}_{w(i,\ell-1)}  \dotsm \tilde{\mathcal M}_{w(i,1)}$ is maintained strictly less than 1 for all $\ell \geq d_k(i)$. \Halmos
\end{proof}

\begin{example} [construction of $\bar{\mathcal{M}}_\ell$ and $\tilde{\mathcal{M}}_\ell$] \label{ex:MatrixM}
	In the case of cyclic polling table with $K =3$, the matrices $\bar{\mathcal{M}}_\ell$, $\ell = 1, 3$, are
	\begin{gather*}
	\bar{\mathcal{M}}_1 =
	\begin{bmatrix}
	1-r_1       & r_1 \frac{\rho_1(1+\alpha)}{1-\rho_1} & r_1 \frac{\rho_1(1+\alpha)}{1-\rho_1}  \\
	0      & 1 & 0  \\
	0       & 0 & 1
	\end{bmatrix}
	, \quad
	\bar{\mathcal{M}}_3 =
	\begin{bmatrix}
	1       & 0 & 0  \\
	0      & 1 & 0  \\
	r_3 \frac{\rho_3 (1+\alpha)}{1-\rho_3}       & r_3 \frac{\rho_3 (1+\alpha)}{1-\rho_3} & 1 - r_3
	\end{bmatrix}.
	\end{gather*}
	In addition, the matrices $\tilde{\mathcal{M}}_\ell$, $\ell = 1, 3$, are
	\begin{gather*}
	\tilde{\mathcal{M}}_1 =
	\begin{bmatrix}
	1-r_1       & r_1 \frac{\rho_2(1+\alpha)}{1-\rho_1} & r_1 \frac{\rho_3(1+\alpha)}{1-\rho_1}  \\
	0      & 1 & 0  \\
	0       & 0 & 1
	\end{bmatrix}
	, \quad
	\tilde{\mathcal{M}}_3 =
	\begin{bmatrix}
	1       & 0 & 0  \\
	0      & 1 & 0  \\
	r_3 \frac{\rho_1(1+\alpha)}{1-\rho_3}       & r_3 \frac{\rho_2(1+\alpha)}{1-\rho_3} & 1 - r_3
	\end{bmatrix}.
	\end{gather*}
\end{example}

\section{The First and Second Moments} \label{sec:FirstAndSecondMoment}
This section is dedicated to the derivation and analysis of the buffer occupancy equations. In Section \ref{ap:FirstMoment}, we show that the first moment of the steady-state queue length at the polling epochs, whose existence follows from the stability condition $\rho < 1$, is the unique solution to the first-order buffer occupancy equations.
In Section \ref{ap:SecondMoment}, we obtain a necessary and sufficient condition for the existence of the second moment, which can be characterized as the unique solution to the second-order buffer occupancy equations.

\subsection{Characterization of the First Moment} \label{ap:FirstMoment}

\begin{theorem} \label{lem:FirstMoment}
	Let $q_k(a_i) := \E\off{Q_k(A_i)} $, $k \in \mathcal K$, $i \in \mathcal I$. Then $(q_k(a_i), k \in \mathcal K, i \in \mathcal I)$ is the unique solution to the following system of $IK$ equations
	\begin{equation} \label{eq:q_e}
	\begin{split}
	q_k(a_{i+1}) &= s_{i} \lambda_k + (1-r_{i}) q_k(a_i), \quad \text{ if } k = p(i) \\
	q_k(a_{i+1})  &= s_{i} \lambda_k + q_k(a_i) +  q_{p(i)}(a_i) r_{i} \lambda_k \frac{1}{ \mu_{p(i)}(1-\rho_{p(i)}) }   ,
	\quad \text{if } k \neq p(i) ,
	\end{split}
	\end{equation}
	where $I + 1 := 1$.
\end{theorem}

\begin{proof}{Proof.}
	Following (\ref{eq:104}), for systems in steady state, the p.g.f.\ $F_{i+1}$ is related to $F_i$ via
	\begin{equation} \label{eq:111}
	\begin{split}
	&F_{(i+1 \text{ mod } I)} \of{z_1, ...,z_K } = \hat R_{i}\off{\sum_{k=1}^K \of{\lambda_k - \lambda_k z_k}} \times \\
	&F_{i}\of{z_1,...,z_{p(i)-1}, (1-r_{i}) z_{p(i)} + r_{i} \hat \theta_{p(i)} \off{\sum_{k=1, \, k\neq p(i)}^K \of{\lambda_k - \lambda_k z_k}  }, z_{p(i)+1},...,z_K} ,
	\end{split}
	\end{equation}
	where we omit the transient time index $m$ in the subscript.
	
	From \eqref{eq:111}, we establish the following system of $IK$ equations for $\of{f_i(k), i \in \mathcal{I}, k \in \mathcal{K}}$ in (\ref{eq:112}) and (\ref{eq:113}) below.
	In particular, taking derivative of (\ref{eq:111}) with respect to $z_k$ and evaluating it $\mathbf z = \mathbf 1$, we get
	\begin{equation} \label{eq:112}
	\begin{split}
	f_{i+1}(k) &= \off{\frac{\partial }{\partial z_k} \hat R_{i}\off{\sum_{j=1}^K \of{\lambda_j - \lambda_j z_j}}}_{\mathbf z = \mathbf 1} + f_i(k) (1-r_{i}) =s_{i} \lambda_k + (1-r_{i}) f_i(k),
	\end{split}
	\end{equation}
	if $k = p(i)$; and
	\begin{align*}
	&f_{i+1}(k) \\
	&= \off{\frac{\partial }{\partial z_k} \hat R_{i} \off{\sum_{j=1}^K \of{\lambda_j - \lambda_j z_j}}}_{\mathbf z = \mathbf 1} + f_i(k) -  f_i(p(i)) r_{i} \lambda_k \hat \theta_{p(i)}^{(1)} \off{\sum_{j=1, j\neq p(i)}^K \of{\lambda_j - \lambda_j z_j}  }_{\mathbf z= \mathbf 1}  \\
	&= s_{i} \lambda_k + f_i(k) -  f_i(p(i))  r_{i} \lambda_k \hat \theta_{p(i)}^{(1)} \off{0  }  \\
	&= s_{i} \lambda_k + f_i(k) +  f_i(p(i))  r_{i} \lambda_k \frac{1}{ \mu_{p(i)}(1-\rho_{p(i)}) }   ,
	\quad \text{if } k \neq p(i).
	\stepcounter{equation}\tag{\theequation}\label{eq:113}
	\end{align*}
	The last equality in \eqref{eq:113} utilizes the fact that for an $M/G/1$ queue, we have
	\begin{equation*}
	-\hat \theta_k^{(1)} \off{0  }= \frac{\E\off{S_k}}{1-\rho_i} = \frac{1}{\mu_k (1-\rho_k)} , \quad \text{for all } k \in \mathcal{K}.
	\end{equation*}
	
	With a change of variable defined by $g_i(k) := f_i(k) / \mu_k$ for $k \in \mathcal{K}$ and $i \in \mathcal{I}$, we get from (\ref{eq:112}) and (\ref{eq:113}) that
	\begin{equation*}
	\begin{split}
	f_{i+1}(k) \frac{1}{\mu_k}
	&= s_{i} \lambda_k \frac{1}{\mu_k} + (1-r_{i}) f_i(k) \frac{1}{\mu_k}, \quad \text{if } k = p(i), \\
	f_{i+1}(k) \frac{1}{\mu_k}
	&= s_{i} \lambda_k \frac{1}{\mu_k} + f_i(k) \frac{1}{\mu_k} +  f_i(p(i))  r_{i} \lambda_k \frac{1}{\mu_{p(i)}(1-\rho_{p(i)})} \frac{1}{\mu_k}  , \quad \text{if } k \neq p(i). \\
	\end{split}
	\end{equation*}
	Therefore,
	\begin{equation} \label{eq:114}
	\begin{split}
	g_{i+1}(k)
	&= s_{i} \rho_k  + (1-r_{i}) g_i(k), \quad \text{if } k = p(i), \\
	g_{i+1}(k)
	&= s_{i} \rho_k  + g_i(k) +  g_i(p(i))  r_{i} \rho_k \frac{1}{1-\rho_{p(i)}}  , \quad \text{if } k \neq p(i). \\
	\end{split}
	\end{equation}

	For $i \in \mathcal{I}$, let $\mathbf {g}_i := \of{g_i(k), k \in \mathcal{K} } \in \RR_+^K$.
	By (\ref{eq:114}), $\mathbf {g}_{i+1}$ can be derived from $\mathbf {g}_{i}$ via an operator $\mathcal T_i: \R_+^{K} \arr \R_+^K$ with the affine relationship
	\begin{equation*} \label{eq:115}
	\mathbf {g}_{i+1} = \mathcal T_i(\mathbf {g}_{i}) = \mathcal A_i \, \mathbf {g}_{i} + \mathcal B_i,
	\end{equation*}
	where $\mathcal A_i$ is an $K \times K$ matrix and $\mathcal B_i \in \RR_+^K$.
	By \eqref{eq:114}, matrix $\mathcal{A}_i$ is characterized as follows: For $j = p(i)$, row $j$ has value $1-r_i$ at the $p(i)$th entry, and value $0$ at all the other entries. For $j \in \mathcal{K}, j \neq p(i)$, row $j$ has value $r_i \rho_j/(1-\rho_{p(i)})$ at the $p(i)$th entry, value $1$ at the $j$th entry, and value $0$ at all the other entries.
	By inspection, $\mathcal A_i$ is identical to the transpose of matrix $\tilde{\mathcal M}_i$ defined in the proof of Theorem \ref{thm:AllMoments}.

	Consider applying the operator $\mathcal T_i$ successively.
	In particular, the composition $\mathcal T_{i+1} \circ \mathcal T_{i} (\mathbf {g}_{i}) $ maps $\mathbf {g}_{i}$ to $\mathbf {g}_{i+2}$, and so forth.
	Then the operator composed over one server cycle $\mathcal T := \mathcal T_{w(i,1)} \circ \mathcal T_{w(i,2)} \circ \dotsm \circ \mathcal T_{w(i,I)}(\mathbf {g}_{i})$ maps $\mathbf {g}_{i}$ to $\mathbf {g}_{i}$, namely, $\mathbf {g}_{i}$ is a fixed point of
	$\mathcal T$.
	Following the same derivation as for $\tilde{\mathcal M}$ in the proof of Theorem \ref{thm:AllMoments},
	we can show that for $\mathcal A := \mathcal A_{i} \mathcal A_{w(i,1)} \dotsm \mathcal A_{w(i,I-1)}$, its spectral radius $\varrho(\mathcal A) < 1$.
	Since $\mathcal T$ is an affine operator mapping from $\R_+^K$ to $\R_+^K$, it follows from Lemma 5.1 in \cite{matveev2016global} that $\mathcal T$ is a contraction map in $\RR_+^K$, and possesses a unique fixed point (i.e., $\mathbf {g}_{i}$), which in turn implies that
	(\ref{eq:112}) and (\ref{eq:113})
	admit a unique solution (in $\RR_+^{IK}$).
	Therefore, it must be the case that $(\E\off{Q_k(A_i)}, k \in \mathcal K, i \in \mathcal I)$ is the unique solution to \eqref{eq:q_e}. \Halmos
\end{proof}

\begin{example} (computing the first moment for cyclic polling table) \label{ex:CyclicBinomial}
	When the polling table is cyclic ($I = K$), the first moment of the steady-state queue length at the polling epochs can be characterized in closed form.
	
	To derive $q_k(a_k)$, note that by (\ref{eq:113}),
	\begin{equation} \label{eq:24}
	\begin{split}
	q_k(a_{\ell + 1}) - q_k(a_{\ell})
	&= s_\ell \lambda_k  + q_{\ell}(a_\ell) r_\ell \lambda_k \frac{1}{\mu_{p(\ell)} (1-\rho_{p(\ell)})} = s_\ell \lambda_k  + q_{\ell}(a_\ell) r_\ell \lambda_k \frac{1}{\mu_{\ell} - \lambda_\ell} ,
	\end{split}
	\end{equation}
	where $K+1 := 1$.
	Summing (\ref{eq:24}) over $\offf{\ell = j , j+1,...,K,1,2,...,k-1}$ for $j > k$, or $\{\ell = j , j+1,...,k-1\}$ for $j < k$, we get
	\begin{equation} \label{eq:25}
	q_k(a_k) - q_k(a_j) = \lambda_k \sum_{\ell=j}^{k-1} s_\ell  + \lambda_k \sum_{\ell=j}^{k-1} q_\ell(a_\ell) r_\ell \frac{1}{\mu_\ell - \lambda_\ell} .
	\end{equation}
	Plugging in $j = k+1$ into (\ref{eq:25}) and using \eqref{eq:112}, we get
	\begin{equation*} \label{eq:26}
	\begin{split}
	&q_k(a_k) - s_k \lambda_k - (1-r_k) q_k(a_k) = \lambda_k \sum_{\ell=1, \ell \neq k}^{K} s_\ell  + \lambda_k \sum_{\ell=1, \ell \neq k}^{K} q_\ell(a_\ell) r_\ell \frac{1}{\mu_\ell - \lambda_\ell} ,
	\end{split}
	\end{equation*}
	from which
	\begin{equation*}
	r_k q_k(a_k) = \lambda_k \sum_{\ell=1}^{K} s_\ell  + \lambda_k \sum_{\ell=1, \ell \neq k}^{K} q_\ell(a_\ell) r_\ell \frac{1}{\mu_\ell - \lambda_\ell} ,
	\end{equation*}
	and thus
	\begin{equation} \label{eq:28}
	q_k(a_k) = \frac{\lambda_k (1-\rho_k)}{r_k} \frac{s}{1-\rho} .
	\end{equation}
	When $r_k =1$ for all $k \in \mathcal{K}$, BEP reduces to the exhaustive policy,  and (\ref{eq:28}) retrieves (4.10a) in \cite{takagi1986analysis}.

	Now, to get $q_j(a_k)$ for $j \neq k$, note that by (\ref{eq:25}), we have
	\begin{equation*}
	q_j(a_j) - q_j(a_k) = \lambda_j \sum_{\ell=k}^{j-1} s_\ell  + \lambda_j \sum_{\ell=k}^{j-1} q_\ell(a_\ell) r_\ell \frac{1}{\mu_\ell - \lambda_\ell} ,
	\end{equation*}
	which implies that
	\begin{equation} \label{eq:23}
	\begin{split}
	q_j(a_k) &= q_j(a_j) - \lambda_j \sum_{\ell=k}^{j-1} s_\ell  - \lambda_j \sum_{\ell=k}^{j-1} q_\ell(a_\ell) r_\ell \frac{1}{\mu_\ell - \lambda_\ell} \\
	&=  \frac{\lambda_j (1-\rho_j)}{r_j} \frac{s}{1-\rho}  - \lambda_j \sum_{\ell=k}^{j-1} s_\ell  - \lambda_j \sum_{\ell=k}^{j-1} \frac{\lambda_\ell (1-\rho_\ell)}{r_\ell} \frac{s}{1-\rho} r_\ell \frac{1}{\mu_\ell - \lambda_\ell} \\
	&= \lambda_j \of{\frac{ (1-\rho_j)}{r_j} \frac{s}{1-\rho} -  \sum_{\ell=k}^{j-1} s_\ell  -  \sum_{\ell=k}^{j-1} \rho_\ell \frac{s}{1-\rho}  } , \quad j \neq k .
	\end{split}
	\end{equation}
	Again, the summation in \eqref{eq:23} is from $\{\ell = k , k+1,...,K,1,2,...,j-1\}$ for $k > j$, and $\{\ell = k , k+1,...,j-1\}$ for $k< j$.
	Similar to the case of $q_k(a_k)$, setting $\mathbf r = \mathbf 1$ (\ref{eq:23}) retrieves (4.10b) in \cite{takagi1986analysis} for the exhaustive policy.
	
	To summarize, the unique solution to the
	$K^2$ equations in \eqref{eq:112} and \eqref{eq:113} is
	\begin{equation*}
	\begin{split}
	&q_k(a_k) = \frac{\lambda_k (1-\rho_k)}{r_k} \frac{s}{1-\rho}, \quad k \in \mathcal K \\
	&q_j(a_k) = \lambda_j \of{\frac{ (1-\rho_j)}{r_j} \frac{s}{1-\rho} -  \sum_{\ell=k}^{j-1} s_\ell  -  \sum_{\ell=k}^{j-1} \rho_\ell \frac{s}{1-\rho}  } , \quad k,j \in \mathcal K, \quad k \neq j .
	\end{split}
	\end{equation*}
	
\end{example}

\subsection{Necessary and Sufficient Condition for the Existence of the Second Moment} \label{ap:SecondMoment}

\begin{theorem} \label{lem:SecondMoment}
	$\E\off{Q_j(A_i) Q_k(A_i)} < \infty$ for all $i \in \mathcal{I}$ and $j, k \in \mathcal{K}$ if and only if
	$\E\off{S_k^2} < \infty$ for all $k \in \mathcal{K}$, and $\E\off{V_i^2} < \infty$ for all $i \in \mathcal{I}$.
\end{theorem}

\begin{proof}{Proof.}
	Utilizing the recursive structure in (\ref{eq:111}) for systems in steady state, we construct a system of $K^2 I$ equations for $\of{f_i(j,k), i \in \mathcal{I}, j, k \in \mathcal{K}}$.
	In particular, taking the partial derivative of (\ref{eq:111}) with respect to $z_j, z_k$ and evaluating at $\mathbf z = \mathbf 1$, we get
	\begin{align*}
	&f_{i+1}(j,k) \\
	&= \left[  \frac{\partial }{\partial z_j} \left[ \of{ \frac{\partial }{\partial z_k}  \hat R_i\off{\sum_{\ell=1}^K \of{\lambda_\ell - \lambda_\ell z_\ell}} }   \times \right. \right.  \\
	&\quad  F_{i}\of{z_1,...,z_{p(i)-1}, (1-r_{i}) z_{p(i)} + r_{i} \hat \theta_{p(i)} \off{\sum_{\ell=1, \ell\neq p(i)}^K \of{\lambda_\ell - \lambda_\ell z_\ell}  }, z_{p(i)+1},...,z_K}  \\
	& \quad + \hat R_i\off{\sum_{\ell=1}^K \of{\lambda_\ell - \lambda_\ell z_\ell}} \times \\
	& \left. \left. \quad  \frac{\partial }{\partial z_k} F_{i}\of{z_1,...,z_{p(i)-1}, (1-r_{i}) z_{p(i)} + r_{i} \hat \theta_{p(i)} \off{\sum_{\ell=1, \ell\neq p(i)}^K \of{\lambda_\ell - \lambda_\ell z_\ell}  }, z_{p(i)+1},...,z_K} \right] \right]_{\mathbf z= \mathbf 1} \\
	&= \off{ \frac{\partial^2 }{\partial z_j \partial z_k}  \hat R_i\off{\sum_{\ell=1}^K \of{\lambda_\ell - \lambda_\ell z_\ell}} }_{\mathbf z= \mathbf 1} + \off{ \frac{\partial }{\partial z_k}  \hat R_i\off{\sum_{\ell=1}^K \of{\lambda_\ell - \lambda_\ell z_\ell}} }_{\mathbf z= \mathbf 1} \times \\
	&\quad  \off{\frac{\partial }{\partial z_j} F_{i}\of{z_1,...,z_{p(i)-1}, (1-r_{i}) z_{p(i)} + r_{i} \hat \theta_{p(i)} \off{\sum_{\ell=1, \ell\neq p(i)}^K \of{\lambda_\ell - \lambda_\ell z_\ell}  }, z_{p(i)+1},...,z_K} }_{\mathbf z= \mathbf 1} \\
	&\quad + \off{ \frac{\partial }{\partial z_j}  \hat R_i\off{\sum_{\ell=1}^K \of{\lambda_\ell - \lambda_\ell z_\ell}} }_{\mathbf z= \mathbf 1} \times \\
	&\quad  \off{\frac{\partial }{\partial z_k} F_{i}\of{z_1,...,z_{p(i)-1}, (1-r_{i}) z_{p(i)} + r_{i} \hat \theta_{p(i)} \off{\sum_{\ell=1, \ell\neq p(i)}^K \of{\lambda_\ell - \lambda_\ell z_\ell}  }, z_{p(i)+1},...,z_K} }_{\mathbf z= \mathbf 1} \\
	&\quad + \off{\frac{\partial^2 }{\partial z_j \partial z_k} F_{i}\of{z_1,...,z_{p(i)-1}, (1-r_{i}) z_{p(i)} + r_{i} \hat \theta_{p(i)} \off{\sum_{\ell=1, \ell\neq p(i)}^K \of{\lambda_\ell - \lambda_\ell z_\ell}  }, z_{p(i)+1},...,z_K} }_{\mathbf z= \mathbf 1}. \stepcounter{equation}\tag{\theequation}\label{eq:29}
	\end{align*}
	There are four summands in the last equation of (\ref{eq:29}), and we characterize each of them as follows.
	
	For now, suppose $j \neq p(i), k \neq p(i)$. First,
	\begin{equation} \label{eq:30}
	\off{ \frac{\partial^2 }{\partial z_j \partial z_k}  \hat R_i\off{\sum_{\ell=1}^K \of{\lambda_\ell - \lambda_\ell z_\ell}} }_{\mathbf z= \mathbf 1}
	= \lambda_j \lambda_k \E\off{V_i^2} .
	\end{equation}
	Second,
	\begin{align*}
	&\off{ \frac{\partial }{\partial z_k}  \hat R_i\off{\sum_{\ell=1}^K \of{\lambda_\ell - \lambda_\ell z_\ell}} }_{\mathbf z= \mathbf 1}  \times \\
	&\quad \off{\frac{\partial }{\partial z_j} F_{i}\of{z_1,...,z_{p(i)-1}, (1-r_{i}) z_{p(i)} + r_{i} \hat \theta_{p(i)} \off{\sum_{\ell=1, \ell\neq p(i)}^K \of{\lambda_\ell - \lambda_\ell z_\ell}  }, z_{p(i)+1},...,z_K} }_{\mathbf z= \mathbf 1}  \\
	&= \lambda_k s_i  \of{f_i(j) +  f_i(p(i))  r_i \hat \theta_{p(i)}^{(1)} \off{0  } (-\lambda_j)} \\
	&= \lambda_k s_i  f_i(j) + \lambda_j \lambda_k s_i f_i(p(i))  r_i \E\off{\Theta_{p(i)}} .
	\stepcounter{equation}\tag{\theequation}
	\end{align*}
	Third,
	\begin{align*}
	&\off{ \frac{\partial }{\partial z_j}  \hat R_i\off{\sum_{\ell=1}^K \of{\lambda_\ell - \lambda_\ell z_\ell}} }_{\mathbf z= \mathbf 1} \times \\
	&\quad \off{\frac{\partial }{\partial z_k} F_{i}\of{z_1,...,z_{p(i)-1}, (1-r_{i}) z_{p(i)} + r_{i} \hat \theta_{p(i)} \off{\sum_{\ell=1, \ell\neq p(i)}^K \of{\lambda_\ell - \lambda_\ell z_\ell}  }, z_{p(i)+1},...,z_K} }_{\mathbf z= \mathbf 1} \\
	&= \lambda_j s_i  \of{f_i(k) +  f_i(p(i))  r_i \hat \theta_{p(i)}^{(1)} \off{0  } (-\lambda_k)} \\
	&= \lambda_j s_i f_i(k) + \lambda_j \lambda_k s_i f_i(p(i))  r_i \E\off{\Theta_{p(i)}} .
	\stepcounter{equation}\tag{\theequation}
	\end{align*}
	Fourth,
	\begin{align*}
	&\off{\frac{\partial^2 }{\partial z_j \partial z_k} F_{i}\of{z_1,...,z_{p(i)-1}, (1-r_{i}) z_{p(i)} + r_{i} \hat \theta_{p(i)} \off{\sum_{\ell=1, \ell\neq p(i)}^K \of{\lambda_\ell - \lambda_\ell z_\ell}  }, z_{p(i)+1},...,z_K} }_{\mathbf z= \mathbf 1} \\
	&= f_i(j, k) +  f_i(p(i),k)  r_i \hat \theta_{p(i)}^{(1)} \off{0  } (-\lambda_j) +  f_i(p(i),j)  r_i \hat \theta_{p(i)}^{(1)} \off{0  } (-\lambda_k) \\
	&\quad +  f_i(p(i),p(i))  r_i^2 \of{\hat \theta_{p(i)}^{(1)} \off{0  }}^2 (-\lambda_j)(-\lambda_k) + f_i(p(i)) r_i \hat \theta_{p(i)}^{(2)} \off{0  } (-\lambda_j)(-\lambda_k) \\
	&= f_i(j, k) +  \lambda_j f_i(p(i),k)  r_i \E\off{\Theta_{p(i)}}  +  \lambda_k f_i(p(i),j)  r_i \E\off{\Theta_{p(i)}}    \\
	&\quad +  \lambda_j\lambda_k f_i(p(i),p(i))  r_i^2 \E\off{\Theta_{p(i)}}^2 + \lambda_j\lambda_k f_i(p(i)) r_i \E\off{\Theta_{p(i)}^2}  . \stepcounter{equation}\tag{\theequation}\label{eq:31}
	\end{align*}
	Plugging in (\ref{eq:30})-(\ref{eq:31}) into (\ref{eq:29}), we have for $j \neq p(i), k \neq p(i)$,
	\begin{equation} \label{eq:32}
	\begin{split}
	f_{i+1}(j,k) &= \lambda_j \lambda_k \E\off{V_i^2} + \lambda_k s_i  f_i(j) + \lambda_j \lambda_k s_i f_i(p(i))  r_i \E\off{\Theta_{p(i)}}  \\
	&\quad + \lambda_j s_i f_i(k) + \lambda_j \lambda_k s_i f_i(p(i))  r_i \E\off{\Theta_{p(i)}} \\
	&\quad + f_i(j, k) +  \lambda_j f_i(p(i),k)  r_i \E\off{\Theta_{p(i)}}  +  \lambda_k f_i(p(i),j)  r_i \E\off{\Theta_{p(i)}}    \\
	&\quad +  \lambda_j\lambda_k f_i(p(i),p(i))  r_i^2 \E\off{\Theta_{p(i)}}^2 + \lambda_j\lambda_k f_i(p(i)) r_i \E\off{\Theta_{p(i)}^2}  .
	\end{split}
	\end{equation}
	Then, using the fact that for an $M/G/1$ queue,
	\begin{equation*}
	\E\off{\Theta_{p(i)}} = - \hat \theta_{p(i)}^{(1)} \off{0  } = \frac{\E\off{S_{p(i)}}}{1-\rho_{p(i)}},
	\quad
	\E\off{\Theta_{p(i)}^2} = \hat \theta_{p(i)}^{(2)} \off{0  } = \frac{\E\off{S_{p(i)}^2}}{(1-\rho_{p(i)})^3},
	\end{equation*}
	we get from (\ref{eq:32}) that for $j \neq p(i), k \neq p(i)$,
	\begin{align*}
	f_{i+1}(j,k) &= \lambda_j \lambda_k \E\off{V_i^2} + \lambda_k s_i  f_i(j) + \lambda_j \lambda_k s_i f_i(p(i))  r_i \frac{\E\off{S_{p(i)}}}{1-\rho_{p(i)}}  \\
	&\quad + \lambda_j s_i f_i(k) + \lambda_j \lambda_k s_i f_i(p(i))  r_i \frac{\E\off{S_{p(i)}}}{1-\rho_{p(i)}} \\
	&\quad + f_i(j, k) +  \lambda_j f_i(p(i),k)  r_i \frac{\E\off{S_{p(i)}}}{1-\rho_{p(i)}}  +  \lambda_k f_i(p(i),j)  r_i \frac{\E\off{S_{p(i)}}}{1-\rho_{p(i)}}    \\
	&\quad +  \lambda_j\lambda_k f_i(p(i),p(i))  r_i^2 \of{\frac{\E\off{S_{p(i)}}}{1-\rho_{p(i)}}}^2 + \lambda_j\lambda_k f_i(p(i)) r_i \frac{\E\off{S_{p(i)}^2}}{(1-\rho_{p(i)})^3}  .
	\stepcounter{equation}\tag{\theequation}\label{eq:35}
	\end{align*}

	Similarly, we get the following systems of equations for $f_{i+1}(j,k)$ for $j = p(i) \neq k $, $k = p(i) \neq j$, and $j = p(i) = k$:
	For $j = p(i) \neq k $,
	\begin{equation} \label{eq:33}
	\begin{split}
	f_{i+1}(j,k) &=   \lambda_{p(i)} \lambda_{k} \E\off{V_i^2} + \lambda_{k} s_i  f_i(p(i)) (1-r_i)  + \lambda_{p(i)} s_i f_i(k) + \lambda_{p(i)} \lambda_k s_i f_i(p(i))  r_i \frac{\E\off{S_{p(i)}}}{1-\rho_{p(i)}}  \\
	&\quad + f_i(p(i),k) (1-r_i) + \lambda_k  f_i(p(i),p(i))  r_i (1-r_i) \frac{\E\off{S_{p(i)}}}{1-\rho_{p(i)}} .
	\end{split}
	\end{equation}
	For $k = p(i) \neq j$ (this case follows from \eqref{eq:33} by symmetry),
	\begin{equation} \label{eq:33_2}
	\begin{split}
	f_{i+1}(j,k) &=   \lambda_{p(i)} \lambda_{j} \E\off{V_i^2} + \lambda_{j} s_i  f_i(p(i)) (1-r_i)   + \lambda_{p(i)} s_i f_i(j) + \lambda_{p(i)} \lambda_j s_i f_i(p(i))  r_i \frac{\E\off{S_{p(i)}}}{1-\rho_{p(i)}}  \\
	&\quad + f_i(j, p(i)) (1-r_i) + \lambda_j  f_i(p(i),p(i))  r_i (1-r_i) \frac{\E\off{S_{p(i)}}}{1-\rho_{p(i)}} .
	\end{split}
	\end{equation}
	For $j = p(i) = k $,
	\begin{equation} \label{eq:34}
	\begin{split}
	f_{i+1}(j,k) &=   \lambda_{p(i)} \lambda_{p(i)} \E\off{V_i^2}  + 2 \lambda_{p(i)} s_i  f_i(p(i)) (1-r_i) + f_i(p(i), p(i)) (1-r_i)^2 .
	\end{split}
	\end{equation}
	The derivation of (\ref{eq:33})--(\ref{eq:34}) is similar to that of \eqref{eq:35} and thus omitted.

	Let $\mathbf{F}_i$ be a $K \times K$ matrix whose elements are $f_i(j, k)$, $i \in \mathcal{I}$, $j, k \in \mathcal{K}$.
	Specifically, whenever {\em the second moments of queue length exist}, the $(k,k)$th element of $\mathbf{F}_i$ is  $\E\off{Q_k(A_i) \of{Q_k(A_i)-1} }$, and the $(j,k)$th element is  $\E\off{Q_j(A_i)Q_j(A_i)}$ for $j \neq k$.
	The affine system defined by \eqref{eq:35}--\eqref{eq:34} calculates $\mathbf{F}_{i+1}$ from $\mathbf{F}_i$. In particular, these equations can be viewed as a transformation $\mathcal{T}_i : \RR_+^{K \times K} \arr \RR_+^{K \times K}$, such that $\mathbf{F}_{i+1} = \mathcal{T}_i \of{\mathbf{F}_i}$, $i \in \mathcal{I}$.
	
	For the following, take without loss of generality that $i = 1$. (The same arguments hold for all $i  \in \mathcal{I} \backslash \{1\}$.)
	Let $\mathcal{T} : \RR_+^{K \times K} \arr \RR_+^{K \times K}$ be the composition
	$\mathcal{T} := \mathcal{T}_{I} \circ \mathcal{T}_{I-1} \circ \dotsm \circ \mathcal{T}_1$. Then, since the system is in steady state,
	$\mathcal{T}(\mathbf{F}_1) = \mathbf{F}_{(I+1) \text{ mod } I} = \mathbf{F}_{1}$,
	so that $\mathbf{F}_1$
	is a fixed point of $\mathcal{T}$. We will show next that $\mathcal{T}$ is a contraction map and possesses a unique fixed point. This implies that the unique fixed point $\mathbf{F}_1$ can be acquired by successively applying $\mathcal{T}$ to any initial condition.

	Note that we can express $f_{i+1}(j , k)$ as
	\begin{equation} \label{eq:53}
	f_{i+1}(j, k) = \sum_{\ell, m} f_i (\ell, m) w_i(\ell, m, j, k) + b_i(j, k) ,
	\end{equation}
	where $w_i(\ell, m,j,k)$ is said to be the weight of $f_i(\ell, m)$ in the element $f_{i+1}(j,k)$, and $b_i(j,k)$ is the sum of all the constant terms in the corresponding equation from \eqref{eq:35}--\eqref{eq:34}.
	Specifically, by \eqref{eq:35}--\eqref{eq:34}, we have
	\begin{align*}
	f_{i+1}(j,k)  =
	\begin{cases}
	f_i(j, k) +  \lambda_j f_i(p(i),k)  r_i \frac{\E\off{S_{p(i)}}}{1-\rho_{p(i)}}  +  \lambda_k f_i(p(i),j)  r_i \frac{\E\off{S_{p(i)}}}{1-\rho_{p(i)}}    \\
	\quad  +  \lambda_j\lambda_k f_i(p(i),p(i))  r_i^2 \of{\frac{\E\off{S_{p(i)}}}{1-\rho_{p(i)}}}^2 + b_i(j,k), \quad \text{for } j \neq p(i), k \neq p(i) \\
	f_i(p(i),k) (1-r_i) + \lambda_k  f_i(p(i),p(i))  r_i (1-r_i) \frac{\E\off{S_{p(i)}}}{1-\rho_{p(i)}} + b_i(j,k) , \\ \quad \text{for } j = p(i) \neq k \\
	f_i(j, p(i)) (1-r_i) + \lambda_j  f_i(p(i),p(i))  r_i (1-r_i) \frac{\E\off{S_{p(i)}}}{1-\rho_{p(i)}} + b_i(j,k) ,  \\ \quad  \text{for } k = p(i) \neq j  \\
	f_i(p(i), p(i)) (1-r_i)^2 + b_i(j,k), \quad \text{for } j = p(i) = k .
	\end{cases}
	\stepcounter{equation}\tag{\theequation}\label{eq:50}
	\end{align*}

	To aid the analysis, we apply a substitution of variable
	\begin{equation*}
	g_i(j,k) := \frac{1}{\mu_j \mu_k} f_i(j,k), \quad b'_i(j,k) := \frac{1}{\mu_j \mu_k} b_i(j,k) , \quad i \in \mathcal{I}, \quad  j, k \in \mathcal{K} .
	\end{equation*}
	Analogously to $\mathbf F_i$, let $\mathbf{G}_i$ be a $K \times K$ matrix whose elements are $g_i(j,k)$, $i \in \mathcal{I}$, $j ,k \in \mathcal{K}$.
	Similarly to $\mathcal{T}_i$, let $\mathcal{T}'_i : \RR_+^{K \times K} \arr \RR_+^{K \times K}$ be the transformation such that $\mathbf{G}_{i+1} = \mathcal{T}'_i \of{\mathbf{G}_i}$, $i \in \mathcal{I}$. Let $\mathcal{T}' := \mathcal{T}'_{I} \circ \mathcal{T}'_{I-1} \circ \dotsm \circ \mathcal{T}'_1$, so that $\mathcal{T}'(\mathbf{G}_1) = \mathbf{G}_{(I+1) \text{ mod } I} = \mathbf G_1$.
	
	Dividing both sides of \eqref{eq:50} by $\mu_j \mu_k$ yields
	\begin{equation*}
	\frac{1}{\mu_j \mu_k} f_{i+1}(j,k)  =
	\begin{cases}
	\frac{1}{\mu_j \mu_k}f_i(j, k) +  \frac{1}{\mu_j \mu_k} \lambda_j f_i(p(i),k)  r_i \frac{\E\off{S_{p(i)}}}{1-\rho_{p(i)}}  + \frac{1}{\mu_j \mu_k} \lambda_k f_i(p(i),j)  r_i \frac{\E\off{S_{p(i)}}}{1-\rho_{p(i)}}    \\
	\quad  + \frac{1}{\mu_j \mu_k} \lambda_j\lambda_k f_i(p(i),p(i))  r_i^2 \of{\frac{\E\off{S_{p(i)}}}{1-\rho_{p(i)}}}^2 + \frac{1}{\mu_j \mu_k} b_i(j,k), \\ \quad \text{for } j \neq p(i), k \neq p(i) \\
	\frac{1}{\mu_j \mu_k} f_i(p(i),k) (1-r_i) + \frac{1}{\mu_j \mu_k} \lambda_k  f_i(p(i),p(i))  r_i (1-r_i) \frac{\E\off{S_{p(i)}}}{1-\rho_{p(i)}}  + \frac{1}{\mu_j \mu_k} b_i(j,k), \\ \quad \text{for } j = p(i) \neq k \\
	\frac{1}{\mu_j \mu_k} f_i(j, p(i)) (1-r_i) + \frac{1}{\mu_j \mu_k} \lambda_j  f_i(p(i),p(i))  r_i (1-r_i) \frac{\E\off{S_{p(i)}}}{1-\rho_{p(i)}} + \frac{1}{\mu_j \mu_k} b_i(j,k) ,  \\ \quad  \text{for } k = p(i) \neq j  \\
	\frac{1}{\mu_j \mu_k} f_i(p(i), p(i)) (1-r_i)^2 + \frac{1}{\mu_j \mu_k} b_i(j,k) , \\ \quad \text{for } j = p(i) = k ,
	\end{cases}
	\end{equation*}
	and
	\begin{equation} \label{eq:54}
	g_{i+1}(j,k)  =
	\begin{cases}
	g_i(j, k) +   g_i(p(i),k)  r_i \frac{\rho_j}{1-\rho_{p(i)}}  +  g_i(p(i),j)  r_i \frac{\rho_k}{1-\rho_{p(i)}}    \\
	\quad  +   g_i(p(i),p(i))  r_i^2 \frac{\rho_j \rho_k}{(1-\rho_{p(i)})^2} + b'_i(j,k) ,  \quad \text{for } j \neq p(i), k \neq p(i) \\
	g_i(p(i),k) (1-r_i) +   g_i(p(i),p(i))  r_i (1-r_i) \frac{\rho_k}{1-\rho_{p(i)}} + b'_i(j,k), \\ \quad \text{for } j = p(i) \neq k \\
	g_i(j,p(i)) (1-r_i) +   g_i(p(i),p(i))  r_i (1-r_i) \frac{\rho_j}{1-\rho_{p(i)}} + b'_i(j,k), \\ \quad \text{for } j = p(i) \neq k  \\
	g_i(p(i), p(i)) (1-r_i)^2 + b'_i(j,k) ,  \quad \text{for } j = p(i) = k .
	\end{cases}
	\end{equation}
	
	We thus arrive at the following analogue of \eqref{eq:53}
	\begin{equation*}
	g_{i+1}(j, k) = \sum_{\ell, m} g_i (\ell, m) \, w'_i(\ell, m, j, k) + b'_i(j, k) , \quad i \in \mathcal{I}, \quad j, k \in \mathcal{K} .
	\end{equation*}
	where $w'_i(\ell, m, j, k)$ is the weight of $g_i(\ell, m)$ in the element $g_{i+1}(j,k)$ in \eqref{eq:54}.
	The total weight of $g_i(\ell, m)$ in the matrix $\mathbf{G}_{i+1}$ is given by
	\begin{equation} \label{eq:55}
	\sum_{j, k} w'_i(\ell, m, j, k) =
	\begin{cases}
	1,   \quad \text{for } \ell \neq p(i), m \neq p(i) \\
	r_i \frac{\rho - \rho_{p(i)} }{1-\rho_{p(i)} } + (1-r_i), \\  \quad \text{for } \ell = p(i) \text{ and } m \neq p(i), \text{ or } \ell \neq p(i) \text{ and } m = p(i) \\
	(1-r_i)^2 + 2 \sum_{k \neq p(i)} r_i(1-r_i) \frac{\rho_k}{1-\rho_{p(i)}} + \sum_{j \neq p(i), k \neq p(i)} r_i^2 \frac{\rho_j \rho_k}{(1-\rho_{p(i)})^2},   \\ \quad \text{for } \ell = p(i), m = p(i) .\\
	\end{cases}
	\end{equation}
	Note from the second and third lines in \eqref{eq:55} that
	\begin{equation*}
	r_i \frac{\rho - \rho_{p(i)} }{1-\rho_{p(i)} } + (1-r_i) = 1 - r_i  \frac{1-\rho }{1-\rho_{p(i)}},
	\end{equation*}
	and
	\begin{equation*}
	\begin{split}
	&(1-r_i)^2 + 2 \sum_{k \neq p(i)} r_i(1-r_i) \frac{\rho_k}{1-\rho_{p(i)}} + \sum_{j \neq p(i), k \neq p(i)} r_i^2 \frac{\rho_j \rho_k}{(1-\rho_{p(i)})^2} \\
	&= \of{r_i \frac{\rho - \rho_{p(i)} }{1-\rho_{p(i)} } + (1-r_i)}^2 \\
	&= \of{1 - r_i  \frac{1-\rho }{1-\rho_{p(i)}}}^2,
	\end{split}
	\end{equation*}
	both of which are equal to 1 if $r_i = 0$ and strictly smaller than $1$ if $r_i > 0$.
	Since
	the service ratios $\mathbf r \in \mathcal{R}$,
	it follows from \eqref{eq:55} that the total weight of $g_1(\ell, m)$
	strictly decreases through the application of $\mathcal{T}' = \mathcal{T}'_{I} \circ \mathcal{T}'_{I-1} \circ \dotsm \circ \mathcal{T}'_1$.
	Given this ``contraction" property of the total weight of each element $g_1(\ell, m)$ through $\mathcal{T}'$, we can apply the same arguments as in \cite{levy1989delay} (specifically, Lemma 3.2 and Theorem 3.4) to show that $\mathcal{T}'$ is indeed a contraction map.

	The derivation above establishes that if
	$\E\off{Q_j \of{A_i } Q_k \of{A_i} } < \infty$ for all $i \in \mathcal{I}$ and $j, k \in \mathcal{K}$,
	then $\mathbf F_1$
	is the unique fixed point of the operator $\mathcal{T}$. We next verify the premise of the aforementioned claim.
	When considering the transformation in \eqref{eq:35}--\eqref{eq:34}, we assume that the system is in steady state,
	namely, $\mathbf{F}_1$ is a fixed point of the operator $\mathcal{T}$.
	Indeed, \eqref{eq:35}--\eqref{eq:34} and $\mathcal{T}$ also apply to systems in transiency, as can be seen in the proof of Proposition \ref{lem:PGFBinExh}. In particular, let the superscript $(m)$ be the index of the $m$th server cycle, and let $\mathbf{F}_i^{(m)}$ denote the cross moment matrix of the queue length at the polling epoch of stage $i$ in the $m$th server cycle. Given $\mathbf{F}_1^{(m)}$,
	$\mathcal{T}(\mathbf{F}_1^{(m)}) = \mathbf{F}_1^{(m+1)}$ for all $m \geq 1$.
	Since $\mathcal{T}$ is a contraction map, $\mathbf{F}_1^{(m)}$ converges through successive application of $\mathcal{T}$.
	The same holds for all $\mathbf{F}_i^{(m)}$, $i \in \mathcal{I} \backslash \{1\}$, and their corresponding operators $\mathcal{T}(i) : \RR_+^K \arr \RR_+^K$ such that $\mathcal{T}(i) (\mathbf{F}_i^{(m)}) = \mathbf{F}_i^{(m+1)}$. (Note that $\mathcal{T}(1) \equiv \mathcal T$.)
	In other words, for any initial queue length,
	$\E\off{Q_j ( A_i^{(m)} ) Q_k (A_i^{(m)}) }$ converges to a proper finite value as $m  \arr \infty$. By Fatou's lemma, it holds that
	\begin{equation*} \label{eq:56}
	\begin{split}
	\lim_{m \arr \infty} \E\off{Q_j (A_i^{(m)} ) Q_k (A_i^{(m)}) }
	&\geq \E\off{\liminf_{m \arr \infty} Q_j (A_i^{(m)} ) Q_k (A_i^{(m)}) }  \\
	&= \E\off{Q_j \of{A_i } Q_k \of{A_i} },   \quad i \in \mathcal{I}, \quad j, k \in \mathcal{K} .
	\end{split}
	\end{equation*}
	
	The condition that $\E\off{S_k^2} < \infty$ for all $k \in \mathcal{K}$, and that $\E\off{V_i^2} < \infty$ for all $i \in \mathcal{I}$, is also necessary in order for the second moment to exist. To see this, suppose for the sake of contradiction that there exists some $k \in \mathcal K$ (or $i \in \mathcal I$) such that $\E\off{S_k^2} = \infty$ (or $\E\off{V_i^2} = \infty$), but $\E\off{Q_j(A_i) Q_k(A_i)} < \infty$ for all $i \in \mathcal{I}$ and $j, k \in \mathcal{K}$.
	It follows from the differentiation of the p.g.f.\ recursion that $(\E\off{Q_j(A_i) Q_k(A_i)}, i \in \mathcal{I}, j, k \in \mathcal{K})$ is a solution to the second-order buffer occupancy equations \eqref{eq:35}--\eqref{eq:34}. However, the right-hand side of \eqref{eq:35}--\eqref{eq:34} is infinite due to $\E\off{S_k^2} = \infty$ (or $\E\off{V_i^2} = \infty$), a contradiction. \Halmos
\end{proof}

\begin{remark} \label{rem:SecondMoment}
In the proof of Theorem \ref{lem:SecondMoment}, we first use a contraction mapping argument to show that the second-order buffer occupancy equations admit a unique solution.
We then note that the same operator (i.e., the contraction map) applies to the transient system, which ``tracks" the second moments of the queue length during the transient period.
The existence of the second and cross moments follows by applying Fatou's lemma, the values of which are necessarily equal to the unique solution to the buffer occupancy equations.
The fact that the buffer occupancy approach applies to the system in the transient period holds for branching-type controls beyond BEP (see, e.g., \cite{whitt2002stochastic,levy1989delay}). Thus, it follows from similar arguments as those in the proof of Theorem \ref{lem:SecondMoment} that the existence of a unique solution to the $p$th-order buffer occupancy equations implies that the $p$th moments exist, in all settings for which the equations are available.
Nevertheless, due to the curse of dimensionality, deriving buffer occupancy equations of order higher than two and proving the existence of a unique solution is prohibitive.
This method is therefore restricted to characterizing, and more fundamentally, proving the existence of only the first few moments.
\end{remark}

\section{Asymptotic Approximation of Moments} \label{ap:AsympMoment}

In this section, we develop a simple approximation scheme for the moments of the steady-state queue length and busy times, which is proved to be asymptotically accurate as the switchover times grow without bound.
To carry out the asymptotic analysis, we consider a sequence of systems indexed by $n \geq 1$,
and append a superscript $n$ to all random variables and processes that scale with $n$.
Let $V^n_i$ denote the switchover time from stage $i$ in system $n$.
Under the large-switchover-time scaling, we keep $\lm_k$ and $\mu_k$ fixed (they do not scale with $n$),
and impose the following assumptions on the sequence of switchover times.

\begin{assumption} \label{assum1}
	$\bar V_i^n := V_i^n/n \Arr  s_{i} $ as $n \arr \infty$. Further,
	$\E\off{V_i^n} = n s_i$\, for all $i \in \mathcal{I}$.
\end{assumption}

We assume that each system in the sequence operates under BEP with the same parameter $\mathbf r \in \mathcal R$, for all $n \geq 1$.
For the $n$th system, we use $Q^n$ to denote the queue length process. In addition, we let $A_i^{n,(m)}$, $D_i^{n,(m)}$, and $B_i^{n,(m)}$ denote the polling epoch, departure epoch, and busy time corresponding to stage $i$ in the $m$th cycle, respectively.

Fluid-scaled quantities are denoted with a bar. In particular,  for the $n$th system, $\bar Q^n(\cdot) := Q^n(n\cdot)/n$ is the fluid-scaled queue length process.
In addition, for
\begin{equation*}
\bar A_1^{(m),n} = A_1^{(m),n} / n \qandq \tilde Q^n(m) = \bar Q^n(\bar A_1^{(m),n}) , \quad m \geq 1,
\end{equation*}
$\tilde Q^n := \{\tilde Q^n(m): m \geq 1\}$ is the fluid-scaled embedded DTMC at the polling epochs of stage $1$ in each cycle.

We assume that the process $Q^n$ is stationary for all $n \geq 1$, namely, $\bar Q^n(0) \deq \tilde Q^n(\infty)$, where $\tilde Q^n(\infty)$ is a generic random variable distributed according to the stationary distribution of $\tilde Q^n$.
Following the convention thus far, we drop the transient time index $(m)$ for steady-state quantities. For example, we let $\bar Q^n (\bar A_i^n)$ be a generic random variable following the distribution of the steady-state queue length at the polling epoch of stage $i$ for the $n$th system.

\subsection{Asymptotic Moments of the Queue Length}

\begin{theorem}
	\label{LEM:ASYMPQ2}
	Assume that Assumption \ref{assum1} holds.
	For $(q_k(a_i), k \in \mathcal K, i \in \mathcal I)$ in \eqref{eq:q_e},
	it holds that
	\begin{enumerate}[(i)]
		\item $\E\off{\bar Q_k^n (\bar A_i^n)} = q_{k}(a_i)$ for all $n \geq 1$, $k \in \mathcal{K}$, $i \in \mathcal{I}$.
		\item  If
		(a) $\E\off{S_k^2} < \infty$ for all $k \in \mathcal{K}$,
		(b) $\E\off{(V_i^n)^2} < \infty$ for all $i \in \mathcal{I}$, $n \geq 1$,
		and (c) $\E\off{(\bar V_i^n)^2} \arr s_i^2$ as $n \arr \infty$ for all $i \in \mathcal{I}$,
		then $\E  \off{Q_k^n (A_i^n) Q_j^n (A_i^n)}  < \infty$ for all $n \geq 1$ and
		$$\lim_{n\tinf} \E\off{\bar Q_k^n (\bar A_i^n) \bar Q_j^n (\bar A_i^n)  } = q_{k}(a_i) q_{j}(a_i),
		\qforallq k, j \in \mathcal{K}, \text{ and } i \in \mathcal{I}.$$
		\item If (a) for each $k\in \mathcal{K}$, there exists $\epsilon_k >0$ such that $\E\off{e^{t S_k}} < \infty$ for all $t \in (-\epsilon_k, \epsilon_k)$,
		(b) $\E\off{e^{t V_i^n}} < \infty$ for all $t \in \RR_+$, $i \in \mathcal{I}$, $n \geq 1$, (c) $\E\off{(\bar V_i^n)^\ell} \arr s_i^\ell$ as $n \arr \infty$ for all $\ell \geq 3$, $i \in \mathcal{I}$,
		then $\E\off{Q_k^n (A_i^n)^\ell} < \infty$ for all $n \geq 1$
		and
		$$\lim_{n\tinf} \E\off{\bar Q_k^n (\bar A_i^n)^\ell} = (q_{k}(a_i))^\ell \qforallq \ell \geq 3, \,\, k \in \mathcal{K}, \text{ and } i \in \mathcal{I}.$$
	\end{enumerate}
\end{theorem}

Theorem \ref{LEM:ASYMPQ2} suggests that the $p$th moment of the steady-state queue length at the polling epochs can be approximated by the $p$th power of the expected steady-state queue length at the polling epochs, and this approximation is asymptotically accurate as the switchover times increase to infinity.
It is significant that this approximation method requires little computation effort, as it only requires solving for the first moment via the system of $IK$ linear equations in \eqref{eq:q_e}; see Example \ref{ex:CyclicBinomial} for a case where closed-form solutions can be obtained.
Later in Section \ref{sec:Numerical}, we demonstrate the effectiveness of this approximation method via simulation.

\begin{proof} {Proof of Theorem \ref{LEM:ASYMPQ2}.}
	We denote by $F_{i}^n\of{z_1, ...,z_K }$ the joint p.g.f.\ of $Q^n$ at time $A_i^n$, namely,
	\begin{equation*}
	F_{i}^n(z_1,...,z_K) := \E\off{\prod_{k=1}^K z_k^{Q_k^n \of{A_i^n} }} , \quad i \in \mathcal{I} .
	\end{equation*}
	Take ${\mathbf z} \in [ 0, 1]^K$.
	
	\paragraph{Proof of (i).}
	Let $k \in \mathcal{K}$ and $i \in \mathcal{I}$, and note that $\E\off{Q_k^n(A_i^n)} < \infty$ for each $n \ge 1$
	due to the stability of the system under BEP.	
	Since the switchover times do not affect the derivation of $y^{(j)}$ and $z_k^{(j)}$ in \eqref{eq:PGF_recursion}, $k \in \mathcal{K}$, $j \geq 0$,
	we have that, by Proposition \ref{lem:PGFBinExh},
	\begin{equation} \label{eq:PGF_dissect}
	\begin{split}
	F_{i}^n\of{z_1, ...,z_K }
	&= \prod_{j=1}^\infty \hat R^n_{w(i,j)}\off{y^{(j-1)} }
	= \prod_{j=1}^\infty \E\off{e^{-y^{(j-1)} V^n_{w(i,j)} }}.
	\end{split}
	\end{equation}	
	Then
	\begin{align}
	\E\off{Q_k^n(A_i^n)}
	&= \off{\frac{\partial F_i^n(z_1,...,z_K)}{\partial z_k} }_{\mathbf z = \mathbf 1}  \nonumber \\
	&= \sum_{j=1}^\infty \off{\frac{\partial \E\off{e^{-y^{(j-1)} V^n_{w(i,j)} } } }{\partial z_k}
		\prod_{\ell \neq j }^\infty \E\off{e^{-y^{(\ell -1)} V^n_{w(i,\ell)} }}   }_{\mathbf z = \mathbf 1} \label{eq:(b)} \\
	&= \sum_{j=1}^{\infty}  \off{\E\off{ \frac{ \partial e^{-y^{(j-1)}V^n_{w(i,j)}}  }{\partial z_k}  }
		\prod_{\ell \neq j }^\infty \E\off{e^{-y^{(\ell -1)} V^n_{w(i,\ell)} }}   }_{\mathbf z = \mathbf 1}  \label{eq:(c)} \\
	&= \sum_{j=1}^{\infty}  -s_{w(i,j)} n \off{\frac{\partial y^{(j-1)}}{\partial z_k} }_{\mathbf z = \mathbf 1} . \label{eq:(d)}
	\end{align}
	The justification for equalities \eqref{eq:(b)} and \eqref{eq:(c)} is relegated to Section \ref{Sec:AsymMomentAux} below.
	It follows from \eqref{eq:(d)} that
	\begin{equation*}
	\begin{split}
	\E\off{\bar Q_k^n(\bar A_i^n)}
	&= \frac{1}{n} \sum_{j=1}^{\infty}  -s_{w(i,j)} n \off{\frac{\partial y^{(j-1)}}{\partial z_k} }_{\mathbf z = \mathbf 1} = \E\off{Q_k(A_i)} = q_{k}(a_i), \\
	\end{split}
	\end{equation*}
	where the last two equalities follow from Theorem \ref{lem:FirstMoment} for the unscaled stochastic system, which is tantamount to taking $n = 1$ in the sequence
	of systems.

	\paragraph{Proof of (ii).}
	
	That $\E\off{Q_k^n(A_i^n)^2} < \infty$ for all $k \in \mathcal{K}$, $i \in \mathcal{I}$ (and $n \ge 1$) follows from Theorem \ref{lem:SecondMoment}.
	We first characterize $\E\off{Q_k^n(A_i^n)^2}$ via the p.g.f.\ of $Q_k^n(A_i^n)$, and then characterize the limit of the cross moments.
	Taking the partial derivative of \eqref{eq:PGF_dissect} with respect to $z_k$ twice gives
	\begin{align}
	&\off{\frac{\partial^2 \of{ F^n_{i}\of{z_1, ...,z_K } }}{\partial z_k^2} }_{\mathbf z= \mathbf 1} \nonumber \\
	&= \sum_{j=1}^{\infty}  \off{\frac{\partial}{\partial z_k}\of{\E\off{ \frac{ \partial e^{-y^{(j-1)}V^n_{w(i,j)}}  }{\partial z_k}  } \prod_{\ell \neq j }^\infty \E\off{e^{-y^{(\ell -1)} V^n_{w(i,\ell)} }}   }}_{\mathbf z = \mathbf 1} \nonumber \\
	&= \sum_{j=1}^{\infty}  \off{ \frac{\partial}{\partial z_k}\E\off{ \frac{ \partial e^{-y^{(j-1)}V^n_{w(i,j)}}  }{\partial z_k}  }  +  \E\off{ \frac{ \partial e^{-y^{(j-1)}V^n_{w(i,j)}}  }{\partial z_k}  } \sum_{\ell \neq j }^\infty  \frac{\partial}{\partial z_k}
		\E\off{e^{-y^{(\ell -1)} V^n_{w(i,\ell)} }}   }_{\mathbf z = \mathbf 1} \label{eq:(f)} \\
	&= \sum_{j=1}^{\infty}   \of{\E\off{(-V^n_{w(i,j)})^2} \of{ \off{\frac{\partial y^{(j-1)}}{\partial z_k} }_{\mathbf z =\mathbf 1} }^2 - s_{w(i,j)} n \off{\frac{\partial^2 y^{(j-1)}}{\partial z_k^2}}_{\mathbf z =\mathbf 1}  } \label{eq:(g)} \\
	&\quad + \sum_{j=1}^{\infty} \sum_{\ell \neq j }^\infty \of{-s_{w(i,j)} n \off{\frac{\partial y^{(j-1)}}{\partial z_k}}_{\mathbf z =\mathbf 1} } \of{-s_{w(i,\ell)} n \off{\frac{\partial y^{(\ell-1)}}{\partial z_k}}_{\mathbf z =\mathbf 1} } . \nonumber
	\end{align}
	Above, equality \eqref{eq:(f)} involves differentiating an infinite product, which is justified similarly to the proof that equality \eqref{eq:(b)} holds;
	equality \eqref{eq:(g)} involves interchanging the order of taking derivative and expectation,
	which follows from similar bounding arguments employed to derive equality \eqref{eq:(c)}; see Section \ref{Sec:AsymMomentAux} below.
	Then	
	\begin{align*}
	&\E\off{Q_k^n(A_i^n)^2} \\
	&= \off{\frac{\partial F_i^n(z_1,...,z_K)}{\partial z_k} }_{\mathbf z= \mathbf 1} + \off{\frac{\partial^2 F_i^n(z_1,...,z_K)}{\partial z_k^2 } }_{\mathbf z= \mathbf 1} \\
	&= \E\off{Q_k^n(A_i^n)} +  \sum_{j=1}^{\infty}   \of{\E\off{(-V^n_{w(i,j)})^2} \of{ \off{\frac{\partial y^{(j-1)}}{\partial z_k} }_{\mathbf z =\mathbf 1} }^2 - s_{w(i,j)} n \off{\frac{\partial^2 y^{(j-1)}}{\partial z_k^2}}_{\mathbf z =\mathbf 1}  } \\
	&\quad + \sum_{j=1}^{\infty} \sum_{\ell \neq j }^\infty \of{-s_{w(i,j)} n \off{\frac{\partial y^{(j-1)}}{\partial z_k}}_{\mathbf z =\mathbf 1} } \of{-s_{w(i,\ell)} n \off{\frac{\partial y^{(\ell-1)}}{\partial z_k}}_{\mathbf z =\mathbf 1} } .
	\end{align*}
	Under the switchover-time scaling, it holds that
	\begin{align*}
	&\E\off{\bar Q_k^n(\bar A_i^n)^2} \\
	&= \frac{1}{n}\E\off{\bar Q_k^n(\bar A_i^n)}  + \frac{1}{n^2}  \sum_{j=1}^{\infty}
	\of{\E\off{(-V^n_{w(i,j)})^2} \of{ \off{\frac{\partial y^{(j-1)}}{\partial z_k} }_{\mathbf z =\mathbf 1} }^2
		- s_{w(i,j)} n \off{\frac{\partial^2 y^{(j-1)}}{\partial z_k^2}}_{\mathbf z =\mathbf 1} }  \\
	&\quad + \frac{1}{n^2} \sum_{j=1}^{\infty} \sum_{\ell \neq j }^\infty \of{-s_{w(i,j)} n
		\off{\frac{\partial y^{(j-1)}}{\partial z_k}}_{\mathbf z =\mathbf 1} } \of{-s_{w(i,\ell)} n \off{\frac{\partial y^{(\ell-1)}}{\partial z_k}}_{\mathbf z =\mathbf 1} }. \\
	&\arr \sum_{j=1}^{\infty}  \of{-s_{w(i,j)} \off{\frac{\partial y^{(j-1)}}{\partial z_k} }_{\mathbf z =\mathbf 1} }^2 \\
	&\quad+ \sum_{j=1}^{\infty} \sum_{\ell \neq j }^\infty \of{-s_{w(i,j)} \off{\frac{\partial y^{(j-1)}}{\partial z_k}}_{\mathbf z =\mathbf 1} }
	\of{-s_{w(i,\ell)} \off{\frac{\partial y^{(\ell-1)}}{\partial z_k}}_{\mathbf z =\mathbf 1} }  \quad  \text{as } n \arr \infty \\
	&= \of{ \sum_{j=1}^{\infty}  -s_{w(i,j)} \off{\frac{\partial y^{(j-1)}}{\partial z_k} }_{\mathbf z =\mathbf 1} }^2 \\
	&= q_{k}(a_i)^2 ,
	\end{align*}
	where the second-to-last equality is justified by the fact that each summand on its left-hand side is non-negative.
	Indeed, it can be observed from the derivation of $y^{(j)}$, e.g., in \eqref{eq:6}, that $\partial y^{(j)}/\partial z_k \leq 0$ for all $j \geq 0$.
	
	For $k \neq j$, $k, j \in \mathcal{K}$, a similar derivation can be applied to get that
	\begin{equation*}
	\E\off{\bar Q_k^n (\bar A_i^n) \bar Q_j^n (\bar A_i^n)  } \arr q_{k}(a_i) q_{j}(a_i) \quad  \text{as } n \arr \infty .
	\end{equation*}
	
	\paragraph{Proof of (iii).}
	
	That $\E\off{Q^n \of{A_i^n}^\ell} < \infty$ for all $\ell \geq 3$ and $i \in \mathcal{I}$ is established in Theorem \ref{thm:AllMoments}.
	For any $\ell \geq 3$, it follows from similar lines of derivation as in the proofs of Assertions {\bf (i)} and {\bf (ii)} above that
	\begin{equation*}
	\begin{split}
	\E\off{\bar Q_k^n (\bar A_i^n)^\ell}
	&= \frac{1}{n^\ell}  \of{ \sum_{j=1}^{\infty} -s_{w(i,j)} n \off{\frac{\partial y^{(j-1)}}{\partial z_k} }_{\mathbf z= \mathbf 1}  }^\ell + o(1)
	\arr (q_{k}(a_i))^\ell \quad \text{as } n \arr \infty . \Halmos
	\end{split}	
	\end{equation*}
\end{proof}

\subsubsection{Proofs of the Auxiliary Results in the Proof of Theorem \ref{LEM:ASYMPQ2}}  \label{Sec:AsymMomentAux}

\begin{proof}{Proof of equality \eqref{eq:(b)}.}
	Let $\Upsilon_{i,m} : [0, 1]^K \arr \RR_+$ be defined as
	\begin{equation*}
	\Upsilon_{i,m}(\mathbf z) := \prod_{j=1}^m \E\off{e^{-y^{(j-1)} V^n_{w(i,j)} }} ,
	\end{equation*}
	for $y^{(j)}$'s calculated recursively from $\mathbf z$ in \eqref{eq:6}.
	Observe that $y^{(j)} \geq 0$ for all $j \geq 0$ for $\mathbf z \leq \mathbf 1$.
	Since $V^n_{w(i,j)} \geq 0$ w.p.1 for all $j \geq 0$, we have
	\begin{equation*}
	\E\off{e^{-y^{(j-1)} V^n_{w(i,j)} }}  \in (0, 1] \quad \text{for all } j \geq 0 ,
	\end{equation*}
	so that
	\begin{equation*}
	\Upsilon_{i,m}(\mathbf z) \geq \Upsilon_{i, m+1}(\mathbf z) \quad \text{for all }  m \geq 1  .
	\end{equation*}
	
	Since $\{\Upsilon_{i,m} : m \geq 1\}$ is a monotone sequence of continuous functions that converge pointwise to $F_i$ over
	the compact set $[0,1]^K$ by \eqref{eq:PGF}, the converges also holds uniformly over this set by Dini's theorem.
	Thus,
	\begin{equation*}
	\begin{split}
	\frac{\partial F_i^n(z_1,...,z_K)}{\partial z_k}
	&= \frac{\partial \prod_{j=1}^\infty \E\off{e^{-y^{(j -1)} V^n_{w(i,j)} }} }{\partial z_k}  \quad
	\text{by }  \eqref{eq:PGF} \\
	&= \sum_{j=1}^\infty \frac{\partial \E\off{e^{-y^{(j-1)} V^n_{w(i,j)} } } }{\partial z_k}
	\prod_{\ell \neq j }^\infty \E\off{e^{-y^{(\ell -1)} V^n_{w(i,\ell)} }}  ,
	\end{split}
	\end{equation*}
	where the second equality follows from the uniform convergence of $\Upsilon_{i,m}$ to $F_i$. \Halmos
\end{proof}

\begin{proof}{Proof of equality \eqref{eq:(c)}.}
	From the derivation of $y^{(j)}$ in \eqref{eq:6}, we observe that $\partial y^{(j)}/\partial z_k$ is continuous in $\mathbf z$ over $[0,1]^K$
	and thus attains a maximum. Let
	\begin{equation*}
	M(j,k) := \max_{\mathbf z  \, \in \, [0,1]^K } \, \offf{ \bigg| \frac{\partial y^{(j-1)}}{\partial z_k} \bigg| } , \quad j \geq 1.
	\end{equation*}
	Then for all $\mathbf z  \in [0,1]^K$,
	\begin{equation*}
	\bigg| \frac{ \partial e^{-y^{(j-1)}V^n_{w(i,j)}}  }{\partial z_k}  \bigg| = \bigg| - V^n_{w(i,j)} \frac{\partial y^{(j-1)}}{\partial z_k} \bigg|
	\leq M(j,k) V^n_{w(i,j)}  .
	\end{equation*}
	Since $\E\off{M(j,k) V^n_{w(i,j)}} < \infty$, it follows from the dominated convergence theorem that
	\begin{equation*}
	\frac{\partial \E\off{e^{-y^{(j-1)} V^n_{w(i,j)} } } }{\partial z_k} = \E\off{ \frac{ \partial e^{-y^{(j-1)}V^n_{w(i,j)}}  }{\partial z_k}  } . \Halmos
	\end{equation*}
\end{proof}

\subsection{Asymptotic Moments of the Busy Time} \label{ap:AsympB}

An analogue of Theorem \ref{LEM:ASYMPQ2} can be obtained for the steady-state busy times.
Recall that, for each stage $i$,
$\Theta_{p(i)}^{(\ell)}$ denotes the busy period ``generated'' by the service of the $\ell$th served customer in queue $p(i)$
(which is the queue being polled at stage $i$), $i \in \mathcal I$, $\ell \geq 1$.

\begin{proposition}  \label{LEM:ASYMPB}
	Assume that Assumption \ref{assum1} holds.
	Consider the corresponding sequence of busy times
	$\{B_i^n : i \in \mathcal{I} , n \ge 1\}$ over a generic stationary cycle.
	For $(q_k(a_i), k \in \mathcal K, i \in \mathcal I)$ in \eqref{eq:q_e}, it holds that
	\begin{enumerate}[(i)]
		\item
		$\E\off{\bar B_i^n} = r_i q_{p(i)}(a_i) \E\off{\Theta_{p(i)}} $ for all $n \geq 1$, $i \in \mathcal{I}$.
		\item  If
		(a) $\E\off{S_k^2} < \infty$ for all $k \in \mathcal{K}$,
		(b) $\E\off{(V_i^n)^2} < \infty$ for all $i \in \mathcal{I}$, $n \geq 1$,
		and (c) $\E\off{(\bar V_i^n)^2} \arr s_i^2$ as $n \arr \infty$ for all $i \in \mathcal{I}$,
		then $\E\off{\of{\bar B_{i}^n}^2} < \infty$ for all $n \geq 1$ and
		\begin{equation*}
		\lim_{n \arr \infty} \E\off{\of{\bar B_{i}^n}^2} \arr \of{r_i q_{p(i)}(a_i) \E\off{\Theta_{p(i)}}}^2, \quad \text{for all } i \in \mathcal{I} .
		\end{equation*}
		\item If (a) for each $k\in \mathcal{K}$, there exists $\epsilon_k >0$ such that $\E\off{e^{t S_k}} < \infty$ for all $t \in (-\epsilon_k, \epsilon_k)$,
		(b) $\E\off{e^{t V_i^n}} < \infty$ for all $t \in \RR_+$, $i \in \mathcal{I}$, $n \geq 1$, and (c) $\E\off{(\bar V_i^n)^\ell} \arr s_i^\ell$ as $n \arr \infty$ for all $\ell \geq 3$, $i \in \mathcal{I}$,
		then $\E\off{\of{\bar B_{i}^n}^\ell} < \infty$ for all $n \geq 1$
		and
		$$\lim_{n\tinf} \E\off{\of{\bar B_{i}^n}^\ell} = \of{r_i q_{p(i)}(a_i) \E\off{\Theta_{p(i)}}}^\ell \qforallq \ell \geq 3, \quad i \in \mathcal{I}.$$
	\end{enumerate}
\end{proposition}

\begin{proof}{Proof.}
	{\em Proof of (i).}
	Let $\{Y_i^{(j)}: j \geq 1\}$ be a sequence of i.i.d.\ Bernoulli random variables with success probability $r_i$,
	indicating whether or not the $j$th customer at stage $i$ is served. If $Y_i^{(j)} = 1$, then a busy period $\Theta_{p(i)}^{(j)}$ is ``generated"
	that is equal in distribution to the busy period of an $M/G/1$ queue corresponding to queue $k$ (that has arrival rate $\lm_k$ and service rate $\mu_k$).	
	Let $\tilde \theta_i$ be the LST of a generic random variable $\Theta_{p(i)} Y_i$ (where we drop the superscript $(j)$ for individual customers), namely,
	\begin{equation*}
	\tilde \theta_i(u) := \E\off{e^{-u \Theta_{p(i)} Y_i }} .
	\end{equation*}
	Then, the LST of
	$\bar B_i^n$ is given by
	\begin{align*}
	\hat \phi (z) &= \E\off{e^{-z\bar B_i^n}} \\
	&= \E\off{e^{- \frac{z}{n} \sum_{j=1}^{Q^n_{p(i)} \of{A_i^n}} \Theta_{p(i)}^{(j)} Y_i^{(j)}  }} \\
	&=\E\off{ \E\off{e^{- \frac{z}{n} \sum_{j=1}^{Q^n_{p(i)}\of{A_i^n}} \Theta_{p(i)}^{(j)} Y_i^{(j)}  }}  \bigg| Q_{p(i)}^n\of{A_i^n} }  \\
	&= \E\off{ \tilde \theta_{i} \off{\frac{z}{n}}^{Q_{p(i)}^n\of{A_i^n} } } \\
	&= G_{i,p(i)}^n \of{ \tilde \theta_{i} \off{\frac{z}{n}} } , \stepcounter{equation}\tag{\theequation}\label{eq:PhiHat}
	\end{align*}
	for $G_{i,p(i)}^n(z) := \E\off{z^{Q^n_{p(i)}\of{A_i^n}}}$, which is analogous to the marginal p.g.f.\ $G_{i,k}$ defined in \eqref{eq:118} (now for the $n$th system).
	
	For $\ell \geq 1$, let $G_{i,p(i)}^{n(\ell)}$ and $\tilde \theta_{p(i)}^{(\ell)} $ denote the $\ell$th order derivative of $G_{i,p(i)}^{n}$ and
	$\tilde \theta_{p(i)}$, respectively.
	Then
	\begin{equation} \label{eq:ExpecB}
	\begin{split}
	\E\off{\bar B_i^n} &= (-1) \off{ G_{i,p(i)}^{n(1)}\of{ \tilde \theta_{p(i)} \off{\frac{z}{n}} } \tilde \theta_{p(i)}^{(1)} \off{\frac{z}{n}} \frac{1}{n}  }_{z=0} \\
	&= \E\off{\bar Q_{p(i)}^n \of{\bar A_i^n} }  \E\off{\Theta_{p(i)}Y_i} \\
	&= q_{p(i)}(a_i) r_i  \E\off{\Theta_{p(i)}} \quad \text{for all } n \geq 1,
	\end{split}
	\end{equation}
	where the last equality follows from Theorem \ref{LEM:ASYMPQ2} and the independence of $\Theta_{p(i)}$ and $Y_i$.
	Note that $\E\off{\bar Q_{p(i)}^n (\bar A_i^n)} < \infty$ in the second line of \eqref{eq:ExpecB} by the stability of the system.

	\paragraph{Proof of (ii).}
	
	Taking the second derivative of \eqref{eq:PhiHat} gives
	\begin{equation} \label{eq:SecondMomentB}
	\begin{split}
	&\E\off{\of{\bar B_{i}^n}^2} \\
	&= \off{ G_{i,p(i)}^{n(2)}\of{ \tilde \theta_{p(i)} \off{\frac{z}{n}} } \of{\tilde \theta_{p(i)}^{(1)} \off{\frac{z}{n}}}^2 \frac{1}{n^2} + G_{i,p(i)}^{n(1)}\of{ \tilde \theta_{p(i)} \off{\frac{z}{n}} } \tilde \theta_{p(i)}^{(2)} \off{\frac{z}{n}} \frac{1}{n^2} }_{z=0} \\
	&= \of{\E\off{\of{\bar Q_{p(i)}^n \of{\bar A_i^n}}^2} - \frac{1}{n}\E\off{\bar Q_{p(i)}^n \of{\bar A_i^n}} }  \E\off{\Theta_{p(i)}Y_i}^2 + \frac{1}{n} \E\off{\bar Q_{p(i)}^n \of{\bar A_i^n} } \E\off{\of{\Theta_{p(i)}Y_i}^2} \\
	&\arr \of{q_{p(i)}(a_i) r_i \E\off{\Theta_{p(i)}} }^2 \quad \text{as } n \arr \infty,
	\end{split}
	\end{equation}
	where the convergence follows from Theorem \ref{LEM:ASYMPQ2} and the independence of $\Theta_{p(i)}$ and $Y_i$. Note that $\E\off{\bar Q_{p(i)}^n \of{\bar A_i^n}^2} < \infty$ in the third line of \eqref{eq:SecondMomentB} is also implied by Theorem \ref{LEM:ASYMPQ2}.

	\paragraph{Proof of (iii).}
	
	Continuing differentiating the LST $\hat \phi$, we get
	\begin{align*}
	&\E\off{\of{\bar B_{i}^n}^\ell} \\
	&= (-1)^\ell \off{G_{i,p(i)}^{n(\ell)}\of{ \tilde \theta_{p(i)} \off{\frac{z}{n}} } \of{ \tilde \theta_{p(i)}^{(1)} \off{\frac{z}{n}}}^\ell \frac{1}{n^\ell}}_{z=0} + o(1) \\
	&= \frac{1}{n^\ell} \E\off{Q_{p(i)}^n\of{A_i^n} \of{Q_{p(i)}^n\of{A_i^n}-1} \dotsm \of{Q_{p(i)}^n\of{A_i^n}-\ell+1} } \E\off{\Theta_{p(i)}Y_i}^\ell + o(1) \\
	&= \E\off{\of{\bar Q_{p(i)}^n \of{\bar A_i^n} }^\ell} \E\off{\Theta_{p(i)}Y_i}^\ell + o(1) \\
	&\arr \of{q_{p(i)}(a_i) r_i  \E\off{\Theta_{p(i)}} }^\ell \quad \text{as } n \arr \infty .
	\stepcounter{equation}\tag{\theequation}\label{eq:49}
	\end{align*}
	Note that $\E\off{Q_{p(i)}^n \of{A_i^n}^j} < \infty$ for all $1 \leq j \leq \ell$ is needed for \eqref{eq:49} to be well-defined, which is implied by Theorem \ref{LEM:ASYMPQ2}. \Halmos
\end{proof}

\subsection{Numerical Experiments} \label{sec:Numerical}
Let $b_i := r_i q_{p(i)}(a_i) \E\off{\Theta_{p(i)}}$, $i \in \mathcal I$. In light of Theorem \ref{LEM:ASYMPQ2} and Proposition \ref{LEM:ASYMPB}, we can approximate $\E\off{Q_k^n(A_i^n)^p}$ by $(n q_k(a_i))^p$ and $\E\off{(B_i^n)^p}$ by $(nb_i)^p$, with an $o(n^p)$ error term, for $(q_k(a_i), k \in \mathcal K, i \in \mathcal I)$ in \eqref{eq:q_e}.
For an unscaled stochastic system in reality, the value of $n$ can be estimated by comparing the magnitude of the service times and switchover times. For example, it is reasonable to set $n = 1$ if the service times and switchover times are of the same order of magnitude, and set $n = 10$ if the switchover times are expected to be approximately $10$ times larger than the service times.

To verify the effectiveness of the approximation, we simulate a three-queue system with the polling table $(1,2,3,2,3)$.
The queues are equipped with exponentially distributed service times and deterministic switchover times, with parameters $\lambda_1 = \lambda_2 =\lambda_3 = 2$, $\mu_1 = \mu_2 = \mu_3 = 8$, and $s_1 = s_2 = s_3 = 2$.
Moreover, the system operates under BEP with parameter $\mathbf r = (1, 0.6, 1, 1, 0.4)$.
We examine the approximation accuracy for different orders of magnitude of the switchover times. In particular, we consider four systems indexed by $n = 1, 10, 100, 1000$. For the $n$th system, the switchover times are set to $V_i^n = n s_i$ for $i \in \mathcal I = \{1,...,5\}$.

For the simulation, we initialize the $n$th system at the polling epoch of stage $1$, with queue length equal to the expected steady-state queue length floored to the nearest integer, i.e., $Q^{n}(0) := \lfloor \E\off{Q^n(A_1)} \rfloor = \lfloor q(a_1) n \rfloor $, where $q(a_1)$ is obtained by solving \eqref{eq:q_e}.
We simulate the system for a total of $100$ cycles and use the mean of $Q_k^n(A_i^n)^p$ recorded at the polling epochs over the simulated time horizon as a ground truth estimate of $\E[Q_k(A_1)^p]$. (The half-width of the 95\% confidence intervals estimated based on the batch means method is very small, i.e., smaller than 1 in most cases.)
Table \ref{fig:AsympMoment} illustrates the accuracy of using $(n q_k(a_i))^p$ to approximate $\E\off{Q_k^n(A_i^n)^p}$ for $p = 1,2,3,4,5$ and $n = 1, 10, 100, 1000$.
We observe that the approximation accuracy worsens for larger orders of moment.
As $n$ increases, the difference between $(n q_k(a_i))^p$  and the simulated $\E\off{Q_k^n(A_i^n)^p}$ diminishes as expected.
Even for a small value of $n = 10$, the difference is already well below $10\%$ of the estimated value in most cases.
Similar trends are observed in Table \ref{fig:BusyTime} which evaluates the effectiveness of using $(nb_i)^p$ to estimate $\E\off{(B_i^n)^p}$.

\begin{table}[h]
	\caption{Approximation for the moments of queue length under BEP}
	\includegraphics[width=15cm]{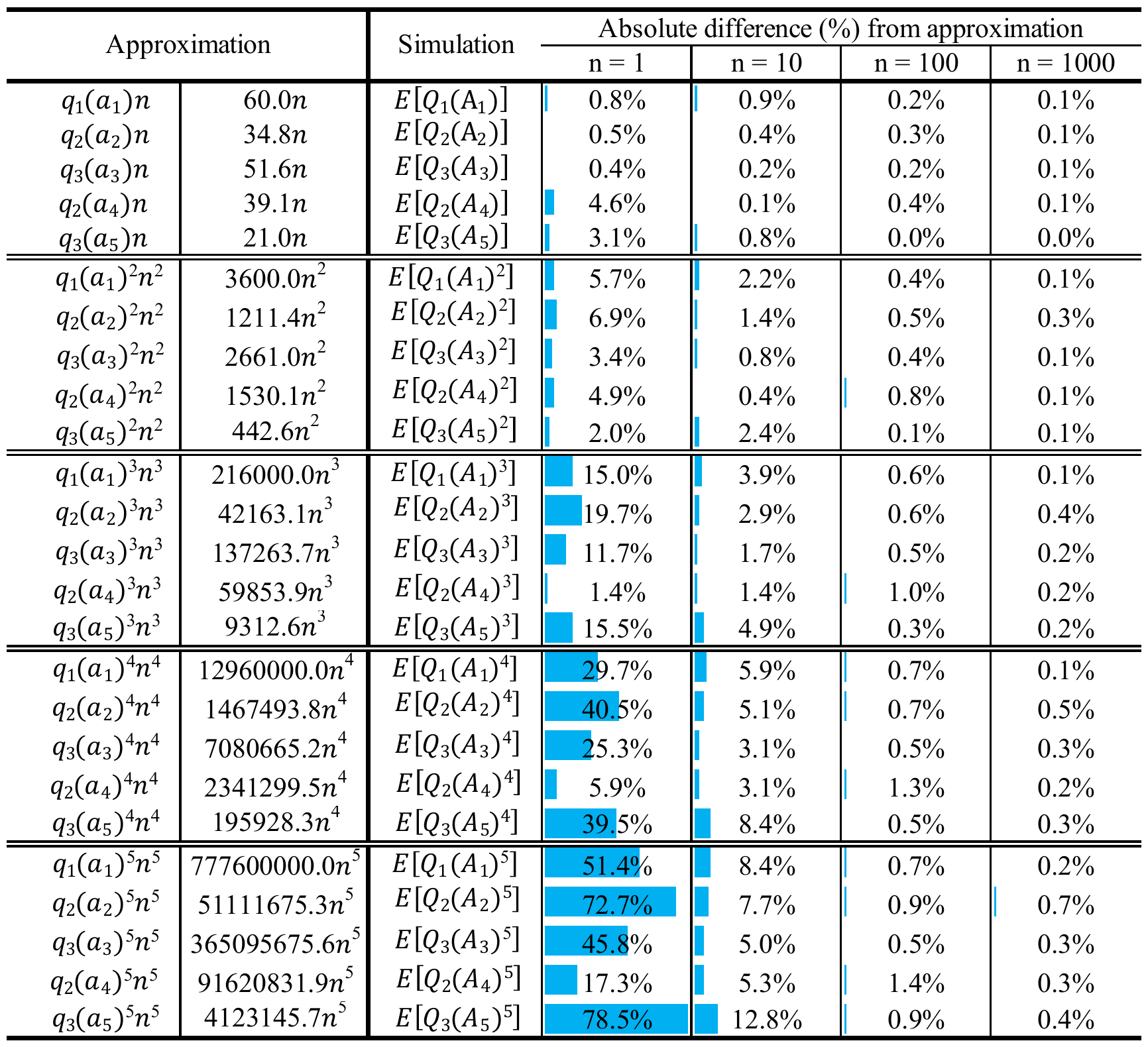}
	\centering
	\label{fig:AsympMoment}
\end{table}

\begin{table}[h!]
	\caption{Approximation for the moments of busy times under BEP}
	\includegraphics[width=15cm]{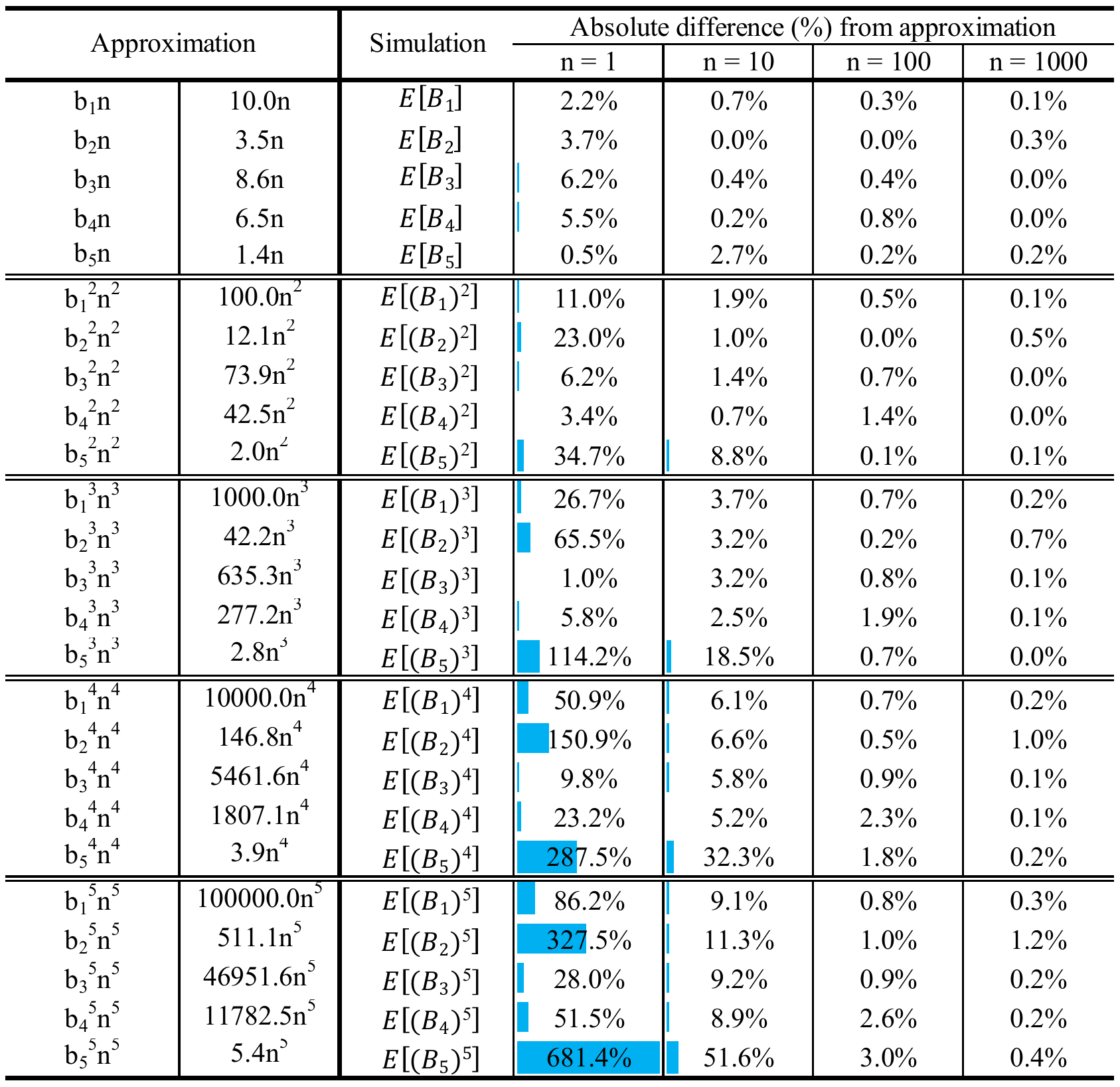}
	\centering
	\label{fig:BusyTime}
\end{table}

\section{Fluid Models with Applications to Other Policies} \label{sec:Fluid}
The approximation scheme for the moments under BEP introduced in Section \ref{ap:AsympMoment} can be summarized as follows:
\begin{enumerate}
	\item Use the buffer occupancy approach to derive the p.g.f.\ for the steady-state queue length at the polling epochs (Proposition \ref{lem:PGFBinExh}).
	\item Differentiate the p.g.f.\ to the get the first-order buffer occupancy equations (Equation \eqref{eq:q_e}), and solve these equations
	to obtain the first moment of the steady-state queue length at the polling epochs $(q_k(a_i), k \in \mathcal K, i \in \mathcal I)$ (Theorem \ref{lem:FirstMoment}).
	\item Use $(q_k(a_i), k \in \mathcal K, i \in \mathcal I)$ to derive approximations for higher moments of the steady-state queue length and busy times
	(Theorem \ref{LEM:ASYMPQ2}, Proposition \ref{LEM:ASYMPB}).
\end{enumerate}

In particular, the values of the elements of $(q_k(a_i), k \in \mathcal K, i \in \mathcal I)$ provide simple approximations for the moments
when the switchover times are  large.
We next demonstrate that $(q_k(a_i), k \in \mathcal K, i \in \mathcal I)$ is related to the equilibrium behavior of the fluid model corresponding to the stochastic system,
under the properly translated fluid control.
Using such fluid-model interpretation, similar approximations can be easily derived for other service policies by
identifying the fluid equilibrium under those controls, even for controls to which the buffer occupancy approach is not applicable,
e.g., policies that do not satisfy the branching-type property \citep{resing1993polling}.

\subsection{The Fluid Model for BEP}

A deterministic fluid model for a stochastic network is achieved by replacing the primitive processes with their rates, together with a control process.
Due to the relationship between fluid models and fluid limits (FWLLN), the control process in the fluid model can be ``translated'' from the stochastic system
by using asymptotic considerations.

In particular, to construct the fluid model we start by replacing the arrival and service processes by their rates, so that
fluid flows into buffer $k$ at a constant rate $\lm_k$, and is flowing out of the buffer at a rate $\mu_k$ {\em whenever the server is attending that buffer}.
The fluid control is then simply the enumeration of the time epochs at which the server is switching from each buffer.
Letting $q_k(t)$ denote the fluid model for queue $k$, namely, it is the fluid content in buffer $k$ at time $t$, we have that
\begin{equation*}
q_k(t) = q_k(0) + \lm_k t - \mu_k \Phi_k(t), \quad t \geq 0, \quad k \in \mathcal K,
\end{equation*}
where $\Phi_k := \{\Phi_k(t) : t\geq 0\}$ is a cumulative process of the form
\begin{equation*}
\Phi_k(t) = \int_{0}^{t} \phi_k(s) ds ,
\end{equation*}
for an indicator function $\phi_k : s \mapsto \{0,1\}$ which is equal to $1$ if the server is attending queue $k$ at time $s$, and to $0$ otherwise.

To characterize $\Phi_k$, we note that, in the stochastic system under BEP, the mean queue length when the server switches away
from queue $k$ at the end of stage $i$ is a proportion $(1-r_i)$ of its value at the polling epoch of stage $i$,
and so the same should be true for the fluid model of this queue.
It is therefore convenient to write those equations in terms of polling epochs. To this end,
for $i \in \mathcal I$ and $m \geq 1$, let $a_i^{(m)}$ denote the polling epoch of stage $i$ in the $m$th cycle.
Recall that $s_i$ denotes the mean switchover time incurred when the server switches from
stage $i$ to stage $i+1$, and $s = \sum_{i \in \mathcal I} s_i$. Then, for $k \in \mathcal K$, $i \in \mathcal I$, and $m \geq 1$,
\begin{align*}
q_k(a_i^{(m+1)}) &= q_k(a_i^{(m)}) + \lambda_k  (a_i^{(m+1)} -a_i^{(m)})  - \sum_{ \{j \in \mathcal I: p(j) = k, j \geq i\} } r_j q_k(a_j^{(m)}) \\
&\quad - \sum_{\{ j \in \mathcal I: p(j) = k, j < i \} } r_j q_k(a_j^{(m+1)}),
\stepcounter{equation}\tag{\theequation}\label{eq:20}
\end{align*}
where the length of the $m$th cycle is given by
\begin{equation}  \label{eq:21}
\bsplit
T^{(m)} & := a_i^{(m+1)} -a_i^{(m)} \\
& = s + \frac{1}{\mu_k - \lambda_k}\of{\sum_{ \{j \in \mathcal I: p(j) = k, j \geq i \}} r_j q_k(a_j^{(m)})
	+ \sum_{ \{j \in \mathcal I: p(j) = k, j < i\} } r_j q_k(a_j^{(m+1)})}.
\end{split}
\end{equation}

To approximate the stationary queue via the fluid model, we assume that the fluid model is itself stationary, in the sense that it is a periodic
function whose period equals the mean stationary cycle length; we refer to the trajectory of a stationary fluid model as a {\em periodic equilibrium} (PE).
Then for all $k \in \mathcal K$, $i \in \mathcal I$, and $m \geq 1$,
\begin{equation}  \label{eq:22}
T^{(m)} = T^{(m+1)} \qandq q_k(a_i^{(m)} + t) = q_k(a_i^{(m+1)} + t), \qforallq t \ge 0.
\end{equation}
By plugging \eqref{eq:21} and \eqref{eq:22} into \eqref{eq:20}, it is easily verified that the balance equations
for the fluid queue in equilibrium evaluated at the polling epochs are precisely the first-order buffer occupancy equations \eqref{eq:q_e}
in Theorem \ref{lem:FirstMoment}. Thus, instead of considering $( q_{k}(a_i), k \in \mathcal K, i \in \mathcal{I} )$
to be the first moment of the steady-state queue length at the polling epochs in the stochastic system,
we can interpret it as the embedded queue length at the polling epochs in the PE.

It is significant that the solution to \eqref{eq:q_e}, interpreted as a stationary fluid model (a PE) for the stochastic system under BEP
{\em in stationarity}, is also the fluid limit of $\bar Q^n := \{Q(nt)/n : t \ge 0\}$ as $n\tinf$
under the large switchover-time asymptotic regime in Assumption \ref{assum1}; see Lemma 5.2
in \cite{Hu20arxiv}.
Thus, Theorem \ref{LEM:ASYMPQ2}(iii) implies that $\{\bar Q^n_k(\bar A_i^n)^p : n \geq 1\}$ is uniformly integrable for each $p$ under the conditions of this theorem, for $k \in \mathcal K$, $i \in \mathcal I$.

\subsection{Applications of Fluid Models to Other Policies}  \label{sec:OtherPolicy}
For policies other than BEP, and in particular, branching-type controls for which the required transforms can be employed,
one can repeat the main arguments in this paper in order to prove analogous results to Theorem \ref{LEM:ASYMPQ2} and Proposition \ref{LEM:ASYMPB} (under appropriate
regularity conditions).
However, the observation that the limits for the moments agree with the corresponding powers of the fluid limit suggests that
the fluid approach can be used to approximate the moments, even if one cannot, or does not wish, to rigorously prove all the required results.
In particular, we propose using fluid models to approximate the moments whenever the switchover times are large as a heuristic
(and remind that the literature on computational tools to compute moments also did not establish that those moments exist).
The advantage of the fluid approximation is its simplicity, especially when the buffer occupancy approach is not applicable.

We now demonstrate the effectiveness of our proposed fluid-model approximations for the moments under two server-switching policies---BGP and BSP---via simulation experiments. For BGP, which like BEP is a branching-type policy,
the first-order buffer occupancy equations coincide with the balance equations for the corresponding fluid model.
However, BSP does not satisfy the branching-type property, and so the buffer occupancy equations cannot be derived for this policy.
Nevertheless, the fluid PE for this latter policy is simple to derive. (It is also standard to show that the PE is achieved as a FWLLN for the
stationary stochastic system under BSP.)

\paragraph{Fluid Approximations Under the Binomial-Gated Policy.} \label{sec:BinGated}
First proposed by \cite{levy1991binomial}, BGP is similar to BEP,
except that the server only serves the customers who are present at the polling epoch of the visit (i.e., excluding those who arrive after the queue is polled).

\begin{definition} [BGP]
	For $\mathbf r \in \mathcal R$, $i \in \mathcal I$, if the server finds $N$ customers at the polling epoch of stage $i$, then the server serves a total of $Y_i(N, r_i)$ customers, where $Y_i(N, r_i)$ is a binomial random variable with parameters $N$ and $r_i$, independently of all other random variables and processes.
\end{definition}

Constructing a fluid model for a stochastic polling system operating under BGP is straightforward:
letting $y \in \RR_+$ be the fluid queue at the polling epoch of stage $i$,
the server serves a proportion $r_i$ of the initial fluid content before switching to the next stage.
Then, as for BEP, one can show that the balance equations to solve for the (unique) PE at the polling epochs are given by, 
\begin{equation} \label{eq:q_e_gated}
\begin{split}
q_k(a_{i+1}) &= s_{i} \lambda_k + (1-r_{i}) q_k(a_i) + \lambda_k r_{i} q_k(a_i) / \mu_{k}, \quad \text{ if } k = p(i) \\
q_k(a_{i+1})  &= s_{i} \lambda_k + q_k(a_i) +  \lambda_k r_{i} q_{p(i)}(a_i) / \mu_{p(i)}  ,
\quad \text{if } k \neq p(i) ,
\end{split}
\end{equation}
for all $k \in \mathcal K$ and $i \in \mathcal I$, where $I + 1 := 1$.
Once again, \eqref{eq:q_e_gated} coincides with the first-order buffer occupancy equations for the first moment of the steady-state queue length at the polling epochs under BGP;
see \citep[Section IV]{levy1991binomial}.

We numerically evaluate the accuracy of using $( q_{k}(a_i)^p, k \in \mathcal K, i \in \mathcal{I} )$ to approximate $\E\off{Q_k(A_i)^p}$, for $p = 1,2,3,4,5$.
Again, we examine four systems with different magnitudes of the switchover times indexed by $n = 1, 10, 100, 1000$,
and with the same system parameters as those for the experiments in Tables \ref{fig:AsympMoment} and \ref{fig:BusyTime}.
For each $n$, the system operates under BGP with the same parameter $\mathbf r$.
Tables \ref{fig:Queue_Gated} and \ref{fig:BusyTime_Gated} summarize the approximation accuracy for the moments of the steady-state queue length and busy times, respectively.
We observe that the approximation error can be relatively large in a few cases with high order of moments and small switchover times (e.g., $p = 4,5$, $n = 1$),
but that the accuracy of the approximations is substantially improved as $n$ increases.

\begin{table}[htb]
	\caption{Approximation for the moments of queue length under BGP}
	\includegraphics[width=15cm]{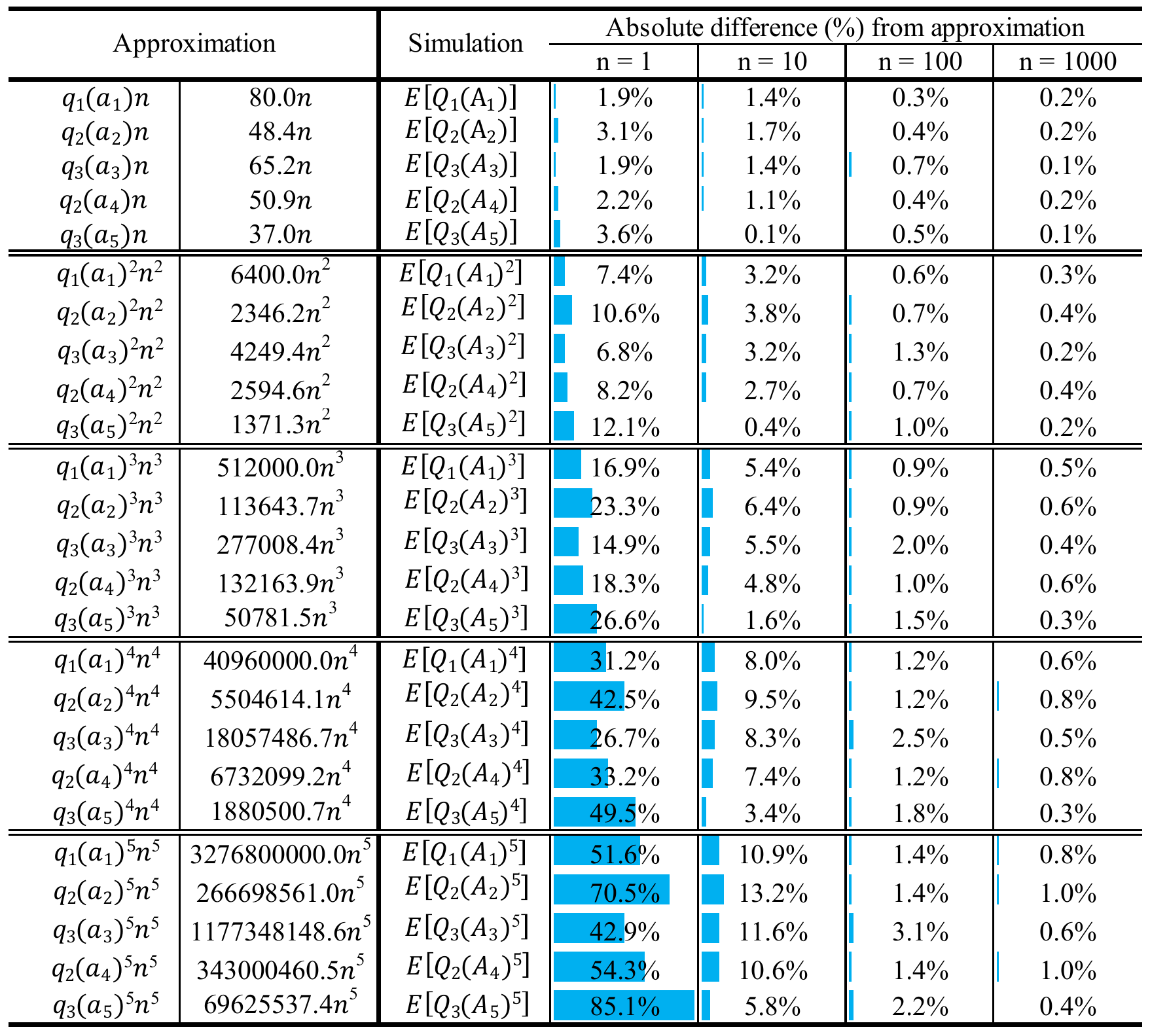}
	\centering
	\label{fig:Queue_Gated}
\end{table}

\begin{table}[htb]
	\caption{Approximation for the moments of busy times under BGP}
	\includegraphics[width=15cm]{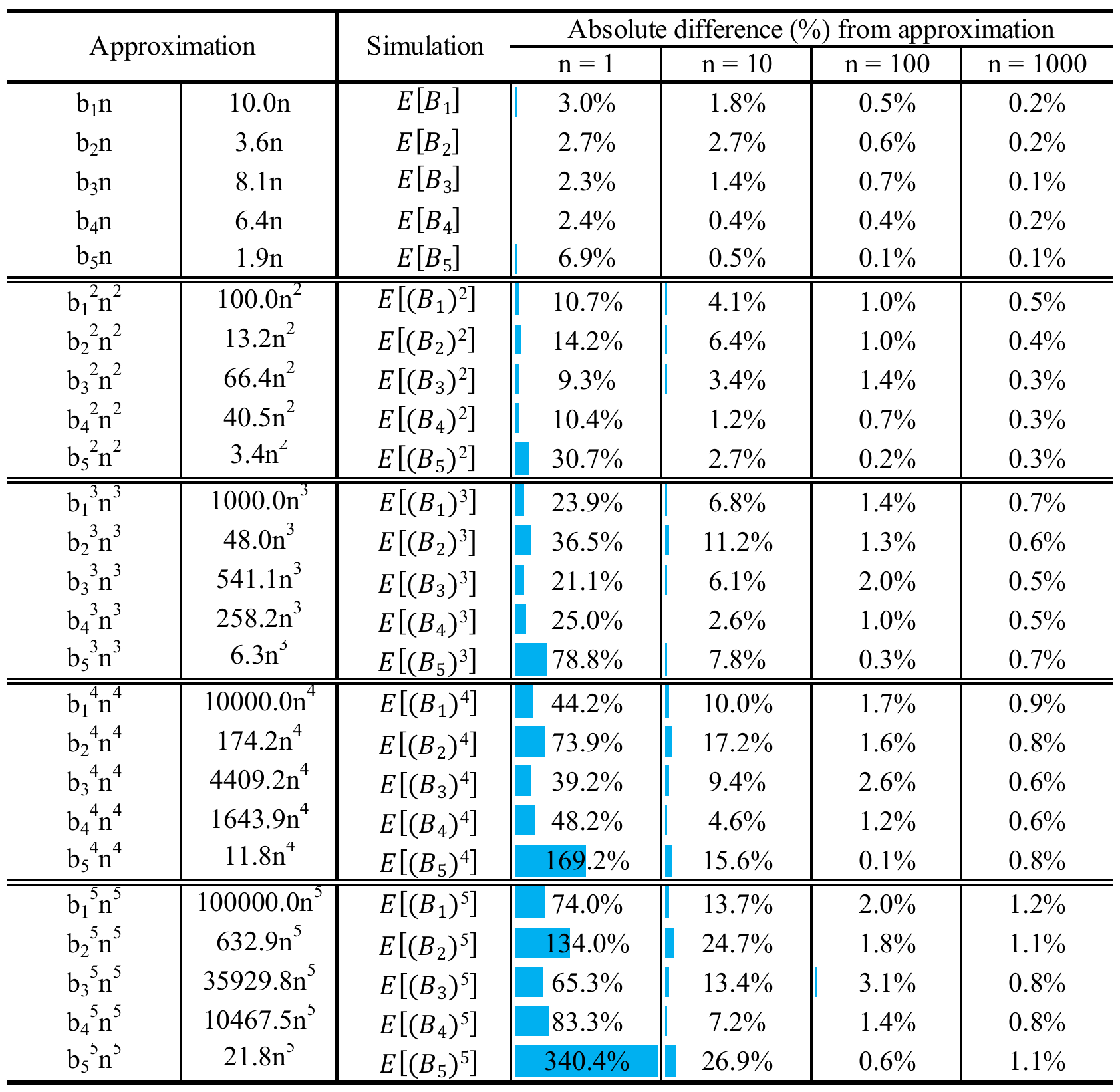}
	\centering
	\label{fig:BusyTime_Gated}
\end{table}

\paragraph{Fluid Approximations Under the Base-Stock Policy.} \label{sec:BaseStock}
Another simple and natural service control is BSP, which imposes a ``serve-up-to level" for each stage.

\begin{definition} [BSP] \label{def:basestock}
	Let $\mathbf Y := (Y_i, i \in \mathcal I) \in \ZZ_+^I$. At stage $i \in \mathcal I$, the server reduces the queue to level $Y_i$ if the queue length at the polling epoch of that stage is larger than $Y_i$. Otherwise, the server switches to the next stage immediately upon the polling epoch of stage $i$.
\end{definition}

Unlike BEP and BGP, the buffer occupancy approach cannot be applied to systems operating under BSP,
unless the same base-stock level is applied for the same queue visited at different stages, i.e., $Y_i = Y_j$ if $p(i) = p(j)$, $i, j \in \mathcal I$.
However, the fluid model is easy to derive in this case.
In particular, we apply the same rule described in Definition \ref{def:basestock} for the fluid model.
We can then show that the PE for the fluid model
at the polling epochs is the unique solution to the balance equations given by
\begin{equation} \label{eq:q_e_basestock}
\begin{split}
q_k(a_{i+1}) &= s_{i} \lambda_k + Y_i , \quad \text{ if } k = p(i) \\
q_k(a_{i+1})  &= s_{i} \lambda_k + q_k(a_i) +  \lambda_k (q_{p(i)}(a_{i}) - Y_{i}) / (\mu_{p(i)} - \lambda_{p(i)}) ,
\quad \text{if } k \neq p(i) ,
\end{split}
\end{equation}
for $k \in \mathcal K$, $i \in \mathcal I$, where $I + 1 := 1$.

To demonstrate the efficacy of the fluid model to approximate the moments when switchover times are large, we consider the same four systems
in Tables \ref{fig:AsympMoment}--\ref{fig:BusyTime_Gated}, except that the $n$th system now operates under BSP with parameter $n \mathbf Y$,
where we take $\mathbf Y = (0, 6,0,0,4)$.
Tables \ref{fig:Queue_BaseStock} and \ref{fig:BusyTime_BaseStock} report the absolute percentage differences
between the fluid-based approximations and simulated moments for the steady-state queue length and busy times.
Consistently with the observation for BEP and BGP, the approximations for the moments under BSP are accurate in the majority of cases,
except for the high-order moments when $n$ is small.

\begin{table}[h]
	\caption{Approximation for the moments of queue length under BSP}
	\includegraphics[width=15cm]{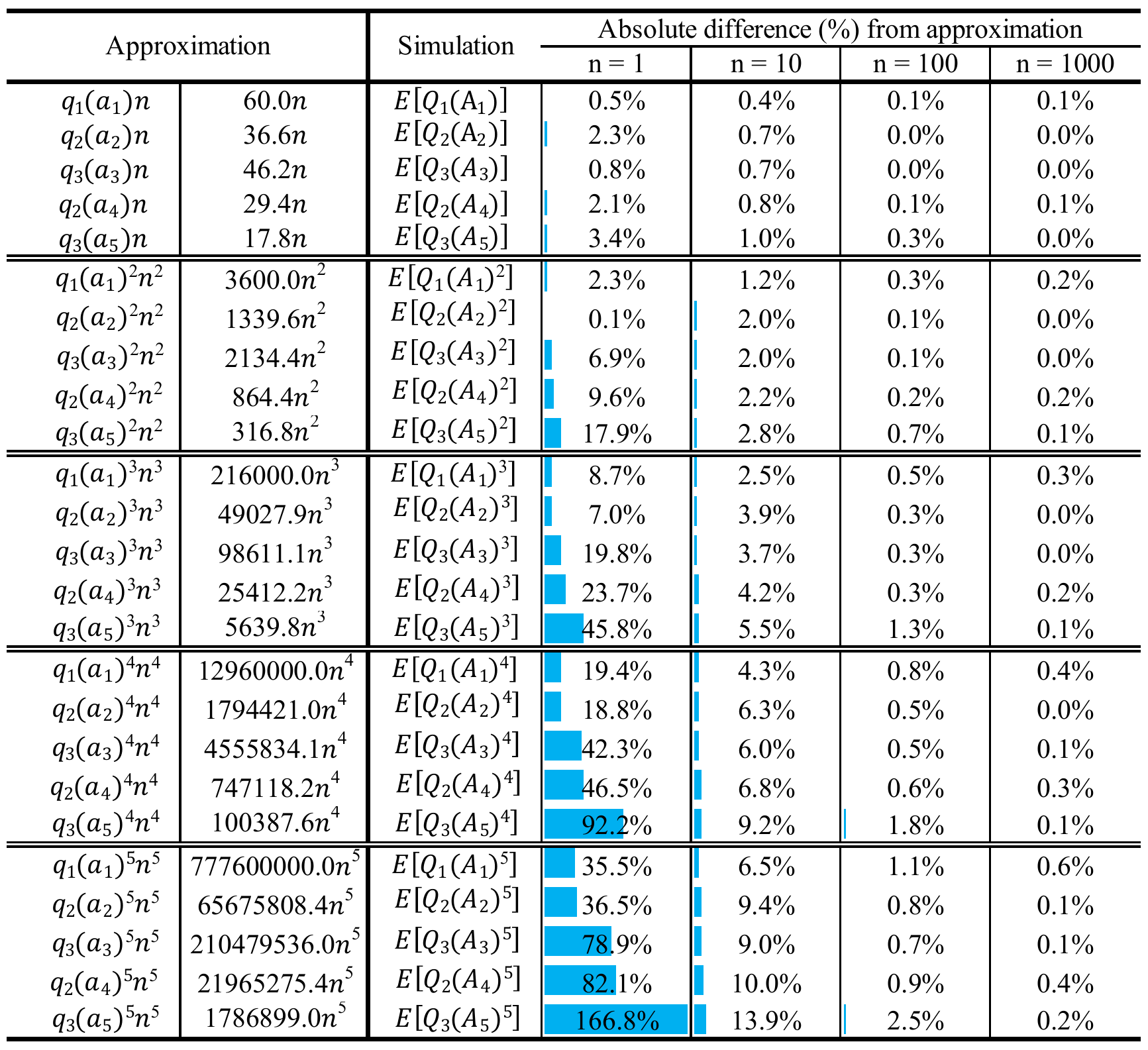}
	\centering
	\label{fig:Queue_BaseStock}
\end{table}

\begin{table}[h]
	\caption{Approximation for the moments of busy times under BSP}
	\includegraphics[width=15cm]{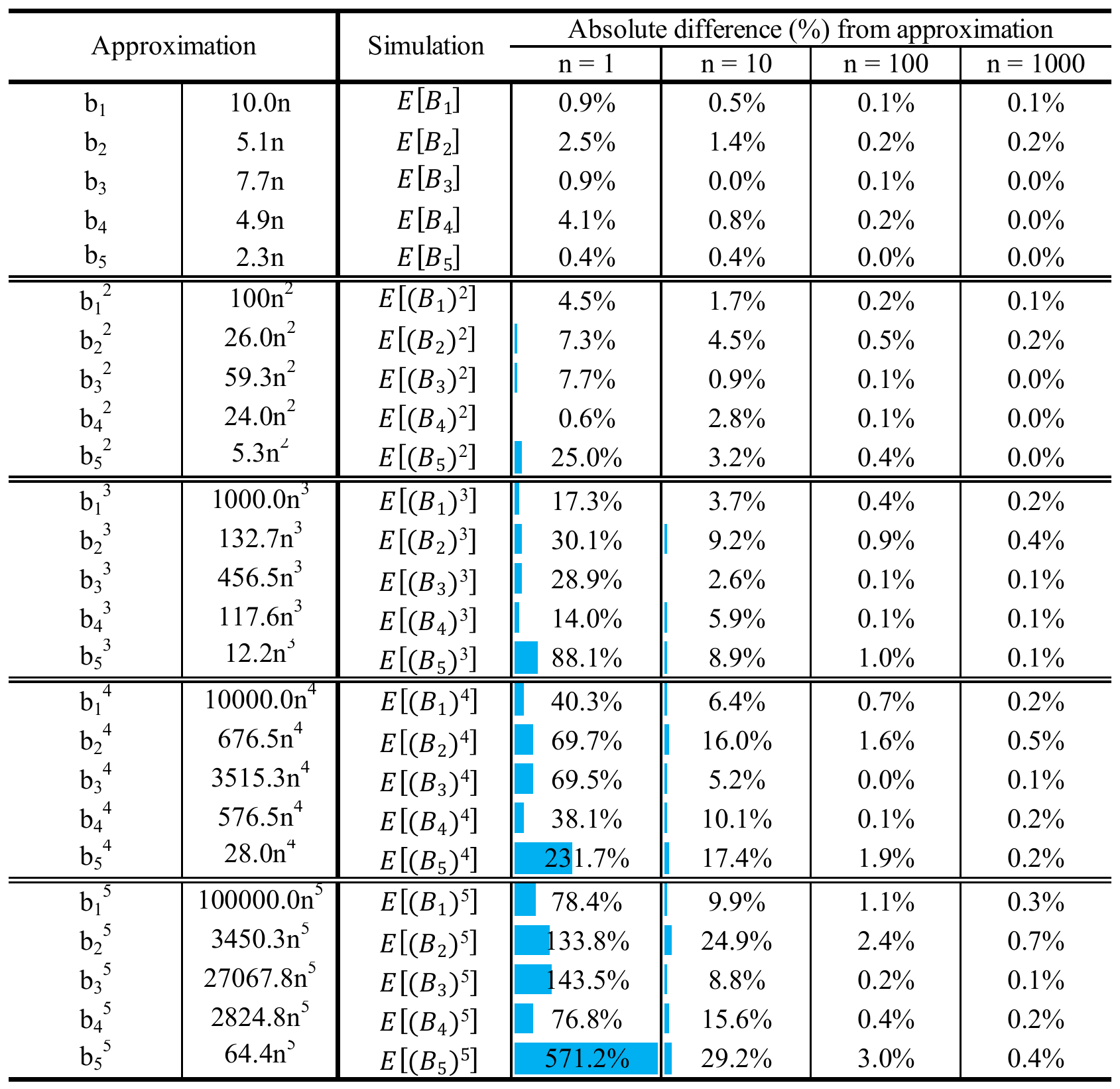}
	\centering
	\label{fig:BusyTime_BaseStock}
\end{table}

\section{Conclusion} \label{sec:Conclusion}
We derived sufficient conditions under which all the moments of the steady-state queue exist,
as well as necessary-and-sufficient conditions for the existence of the second moment, for polling systems operating under BEP.
In particular, if the service-time distributions of all the queues possess a finite m.g.f.\ in some neighborhood of zero,
and if the distribution of each switchover time has a finite m.g.f.\ on the positive real line,
then the steady-state queue length at polling epochs has finite moments of all orders.
For the existence of the second moment, the necessary-and-sufficient condition is that the service time and switchover time distributions have finite variances.
We then showed that, under a large switchover scaling, the $p$th moment of the ``fluid-scaled'' queues converges to the value of the first moment (when it exists) raised to the
$p$th power. Thus, the first moment can be used to approximate higher moments whenever the switchover times are sufficiently large.
Finally, we showed that the the first moment of the queue at polling epochs agrees with the value of the fluid model for the system,
which arises as a FWLLN for the sequence of stationary systems under the large-switchover-time scaling.
This relation suggests that the fluid models can be used to approximate the moments for other controls as well,
either by following the same rigorous steps taken in this paper, or as a heuristic.
The efficacy of the fluid-based approximations for polling systems operating under BEP, BGP, and BSP were demonstrated via numerical experiments.




\bibliographystyle{ormsv080_abv}
\bibliography{bibfile}

\end{document}